\documentclass[11pt]{amsart}
\usepackage{latexsym, amsmath, amssymb,amsthm,amsopn,amsfonts}
\usepackage{version}
\usepackage{epsfig,graphics,color,graphicx,graphpap}
\usepackage{amssymb}
\usepackage{todonotes}
\usepackage{listings}
\usepackage{mathrsfs}
\usepackage[citecolor=blue]{hyperref}
\usepackage{cancel}

\usepackage[framemethod=tikz]{mdframed}
\newcommand{\redsout}{\bgroup\markoverwith{\textcolor{red}{\rule[0.5ex]{2pt}{.4pt}}}\ULon}

\newcommand{\LV}{\left|}
\newcommand{\RV}{\right|}
\newcommand{\LC}{\left(}
\newcommand{\RC}{\right)}

\newcommand{\p}{\partial}

\newcommand{\tx}{\langle \tau,\xi\rangle  }

\def \supp{ \mathrm{supp }}
\def \Om{\Omega}

\def \CHT{C([0,T],H^2(\Om))}
\def \CLT{C^1([0,T],L^2(\Om))}
\def \HLT{H^1(0,T;L^2(\Om))}

\numberwithin{equation}{section}
\newtheorem{theorem}{Theorem}[section]
\newtheorem{corollary}[theorem]{Corollary}
\newtheorem{proposition}{Proposition}[section]
\newtheorem{lemma}{Lemma}[section]

\newtheorem{remark}{Remark}[section]
 
\newcommand{\R}{\mathbb R}

 \definecolor{Magenta}{cmyk}{0,1,0,0}

\definecolor{Lightred}{rgb}{1,0.89,0.9}

\usepackage{url}
\definecolor{mycolor}{rgb}{0.122, 0.435, 0.698}
\definecolor{aliceblue}{rgb}{0.94, 0.97, 1.0}

\newmdenv[innerlinewidth=0.5pt, roundcorner=4pt,linecolor=mycolor,innerleftmargin=6pt,
innerrightmargin=6pt,innertopmargin=6pt,innerbottommargin=6pt]{mybox}
\newmdenv[backgroundcolor=aliceblue,innerlinewidth=0.5pt, roundcorner=4pt,linecolor=mycolor,innerleftmargin=6pt,
innerrightmargin=6pt,innertopmargin=6pt,innerbottommargin=6pt]{mybox1}
\author[Lai]{Ru-Yu Lai}
\address{School of Mathematics, University of Minnesota, Minneapolis, MN 55455, USA}
\curraddr{}
\email{rylai@umn.edu}
 
\author[Lu]{Xuezhu Lu}
\address{Department of Mathematics, Northeastern University, Boston, MA 02115, USA}
\curraddr{}
\email{lu.xuez@northeastern.edu}

\author[Zhou]{Ting Zhou}
\address{School of Mathematical Sciences, Zhejiang University, Hangzhou, China}
\curraddr{}
\email{ting$\_$zhou@zju.edu.cn}


\title{Partial Data Inverse Problems for the Nonlinear Time-Schr\"odinger Equation} 
\begin{document}

\maketitle
\tableofcontents

\begin{abstract}
In this paper we prove the uniqueness and stability in determining a time-dependent nonlinear coefficient $\beta(t,x)$ in the Schr\"odinger equation $(i\p_t+\Delta+q(t,x))u+\beta u^2=0$, from the boundary Dirichlet-to-Neumann (DN) map. In particular, we are interested in the partial data problem, in which the DN-map is measured on a proper subset of the boundary.  We show two results: a local uniqueness of the coefficient at the points where certain type of geometric optics (GO) solutions can reach; and a stability estimate based on the unique continuation property for the linear equation. 
\end{abstract}

\section{Introduction}
We investigate a partial data inverse problem for the time-dependent Schr\"odinger equation with a nonlinear term, for example, in modeling the recovery of the nonlinear electromagnetic second order polarization potential from the partial boundary measurements of electromagnetic fields. 
Let $\Omega\subset\mathbb{R}^n$, $n\geq 3$ be a bounded and convex domain with smooth boundary $\p\Omega$. For $T>0$, we denote $Q:=(0,T)\times \Omega$ and $\Sigma:=(0,T)\times \p\Omega$. 
Suppose $\Gamma$ is an open proper subset of the boundary $\p\Omega$ and denote
$$
\Sigma^\sharp : =(0,T)\times \Gamma.
$$
For $q(t,x)\in C^\infty(Q)$ and $\beta(t,x)\in C^\infty(Q)$, 
we consider the nonlinear dynamic Schr\"odinger equation
\begin{align}\label{IBVP nonlinear}
	\left\{\begin{array}{rcll}
		\LC i \p_t +\Delta + q(t,x)\RC u(t,x) + \beta(t,x)u(t,x)^2 &=& 0 &\quad \hbox{on }Q,\\
		u(t,x)&=&f &\quad \hbox{on }\Sigma,\\
		u(t,x)&=&0 &\quad \hbox{on }\{0\}\times \Omega,\\
		\end{array} \right. 
\end{align} 
where  $\Delta u:=\sum^n_{j=1}{\p^2\over\p x_j^2}$ is the spatial Laplacian.
 
Based on the well-posedness result in Proposition~\ref{prop:forward nonlinear}, the Dirichlet-to-Neumann (DN) map $\Lambda_{q,\beta}$ is well-defined by
\begin{align*}
\Lambda_{q,\beta}:  
  f &\mapsto \p_\nu u\big|_{\Sigma^\sharp},\qquad \quad f\in \mathcal{S}_\lambda(\Sigma)
\end{align*}
for $\lambda>0$ sufficiently small (see \eqref{eqn:S_eps} for the definition of $\mathcal S_\lambda(\Sigma)$, where $\p_\nu u:={\p u\over \p\nu}$ and $\nu(x)$ is the unit outer normal to $\p\Omega$ at the point $x\in \p\Omega$. The inverse problem we consider in this paper is the determination of the nonlinear potential $\beta(t,x)$ from the partial DN-map $\Lambda_{q,\beta}$.

\subsection{Main results}

For a set $B\subset \Omega$, we denote $\mathcal{M}_{B}$ by
$$
\mathcal{M}_{B}:=\{g\in C^\infty(Q):\, \|g\|_{C^\infty(Q)}\leq m_0, \quad \hbox{and } g=0\hbox{ on }(0,T)\times B\}.
$$
for some positive constant $m_0$.
Let $\mathcal{O}\subset \Omega$ be an open neighborhood of the boundary $\p\Omega$ and $\mathcal{O}'\subset \Omega$ be an open neighborhood of $\Gamma^c:=\p\Omega\setminus\Gamma$.  \\

We define an open subset $\Omega_\Gamma$ of $\Omega$ as
\begin{align}\label{DEF:Omega Gamma}
 \Omega_\Gamma:=\big\{p\in\Omega~: \LC\LC \gamma_{p,\omega_1}\cup\gamma_{p,\omega_2}\cup\gamma_{p,\omega_1+\omega_2}\RC  \cap\partial\Omega\RC \subset\Gamma \textrm{ for some } \omega_1,\omega_2\in \mathbb S^{n-1}, \omega_1\perp\omega_2\big\},
\end{align}
where $\gamma_{p,\omega}$ denotes the straight line through a point $p$ in a direction $\omega$ in $\R^n$ and $\mathbb{S}^{n-1}$ is a unit sphere at the origin.
 
Our main results are stated as follows:
\begin{theorem}[Local uniqueness]\label{thm:local uniqueness}
	Assume $q$ and $\beta_j$ are in $C^{\infty}(Q)$ for $j=1,\,2$.  Suppose $\Lambda_{q,\beta_1}(f)=\Lambda_{q,\beta_2}(f)$ for all $f\in \mathcal{S}_\lambda(\Sigma)$ with support satisfying $\text{supp}(f)\subset \Sigma^\sharp$. Then $\beta_1(t,x)=\beta_2(t,x)$ for all $(t,x) \in(0,T)\times \Omega_\Gamma.$
\end{theorem}

\begin{theorem}[Stability estimate]\label{thm:stability}
Assume $\beta_j \in C^{\infty}(Q)$ for $j=1,\,2$. Suppose that $(q,\beta_1-\beta_2)\in \mathcal{M}_{\mathcal{O}}\times\mathcal{M}_{\mathcal{O}}$. Let $\Lambda_{q,\beta_j}:\mathcal{S}_\lambda(\Sigma)\rightarrow L^2(\Sigma^\sharp)$ be the Dirichlet-to-Neumann maps of the nonlinear Schr\"odinger equation \eqref{IBVP nonlinear} associated with $\beta_j$ for $j=1,\,2$.  
There exists a sufficiently small $\delta_0>0$ so that if the DN maps satisfy
\begin{align*} 
\|(\Lambda_{q,\beta_1}-\Lambda_{q,\beta_2})f\|_{L^2(\Sigma^\sharp)}\leq \delta \qquad \textrm{ for all }f\in \mathcal{S}_\lambda(\Sigma),
\end{align*}
for some $\delta\in (0,\delta_0)$, then for any $0<T^*<T$, there exist constants $C>0$  independent of $\delta$  
and $0<\sigma<1$ such that the following stability estimate holds:
\begin{align*}
    \|\beta_1-\beta_2\|_{L^2((0,T^*)\times \Omega)}\leq C \LC \delta^{1\over 12}+|\log(\delta)|^{-\sigma}\RC.
\end{align*}
\end{theorem}
The logarithmic type stability estimate here is expected since we only take measurements on partial region of the boundary of the domain. 
 
The uniqueness result of Theorem~\ref{thm:unique} follows directly from Theorem~\ref{thm:local uniqueness} and Theorem~\ref{thm:stability} by letting $\delta\rightarrow 0$. In particular, due to Theorem~\ref{thm:local uniqueness}, the assumption of $\beta_1-\beta_2$ can be relaxed to $\mathcal{M}_{\mathcal{O}'}$.
\begin{theorem}[Global uniqueness]\label{thm:unique}
Assume $\beta_j \in C^{\infty}(Q)$ for $j=1,\,2$. Suppose that $(q,\beta_1-\beta_2)\in \mathcal{M}_{\mathcal{O}}\times\mathcal{M}_{\mathcal{O}'}$. Let $\Lambda_{q,\beta_j}:\mathcal{S}_\lambda(\Sigma)\rightarrow L^2(\Sigma^\sharp)$ be the Dirichlet-to-Neumann maps of the nonlinear Schr\"odinger equation \eqref{IBVP nonlinear} with $\beta_j$ for $j=1,\,2$. 
If $\Lambda_{q,\beta_1}(f)=\Lambda_{q,\beta_2}(f)$ for all $f\in \mathcal{S}_\lambda(\Sigma)$,
then
$$
\beta_1=\beta_2\quad \hbox{ in }Q.
$$
\end{theorem}

\par
The nonlinear Schr\"odinger equation (NLS) in \eqref{IBVP nonlinear} can be used to model a basic second harmonic generation process in nonlinear optics. A similar NLS is the Gross-Pitaevskii (GP) equation 
\[(i \p_t +\Delta + q) u + \beta(t,x)|u|^2u=0\]
for the single-atom wave function, used in a mean-field description of Bose-Einstein condensates. See \cite{Mal} for discussions of various NLS models based on integrability and existence of stable soliton solutions, such as the nonlinear term of a saturable one, $|u|^2(1+|u|^2/u_0^2)^{-1}$ with $u_0$ a constant, or $(|u|^2-|u|^4)u$. We remark in Remark \ref{rmk:higher_nonlinearity} that our approach can be generalized to power type nonlinearity other than quadratic ones. Similar discussions can be found in \cite{LOST2022} for the GP equation. \\

Similar to those of hyperbolic equations, results related to the determination of coefficients for dynamic Schr\"odinger equations are usually classified into two categories of time-independent and time-dependent coefficients. For the linear equation, stability estimates for recovering the time-independent electric potential or the magnetic field from the knowledge of the dynamical Dirichlet-to-Neumann map were shown in 
\cite{AM, Bella, BC, BD, BKS, CS}. 
A vast literature is devoted for the inverse problems associated to the stationary Schr\"odinger equation, known under the name of Calder\'on problem, see \cite{Na,sylvester1987global} for the major results when the DN-map is measured on the whole boundary and see \cite{FKSjU09, ferreira2009limiting,Dos2009, kenig2007calderon} when measured on part of the boundary. The paper \cite{Eskin} by Eskin is known to be the first to show the unique determination of time-dependent electric and magnetic potentials of the Schr\"odinger equation from the DN-map. Stability for the inverse problem with full boundary measurement was shown in \cite{KianSoccorsi,KianTetlow, CKS}. The stable determination of time-dependent coefficients appearing in the linear Schr\"odinger equation from partial DN map is then given in \cite{Bella-Fraj2020}. The stability estimate for the problem of determining the time-dependent zeroth order coefficient in a parabolic equation from a partial parabolic Dirichlet-to-Neumann map can be found in \cite{ChoulliKian2018}.

In dealing with the inverse problems for nonlinear PDEs, the first order linearization of the DN-map was introduced in recovering the linear coefficient for the medium, and sometimes the nonlinear coefficients. See \cite{Sun2002, Isakov93, victor01, victorN, isakov1994global, SunUhlmann97} for demonstrations for certain semilinear, quasilinear elliptic equations and parabolic equations. 
Recently the higher order linearization, also called the multifold linearization, of the measurement operators (e.g., the Dirichlet-to-Neumann map or the source-to-solution map) has been applied in determining nonlinear coefficients in more general nonlinear differential equations. 
For example, based on the scheme, the nonlinear interactions of distorted plane waves were analyzed to recover the metric of a Lorentzian space-time manifold and nonlinear coefficients using the measurements of solutions to nonlinear hyperbolic equations \cite{KLU2018, LUW2018, UW2020}. In contrast the underlying problems for linear hyperbolic equations are still open, see also \cite{CLOP,LUW2018} and the references therein.  
The method is also applied to study elliptic equations with power-type nonlinearities, including stationary nonlinear Schr\"odinger equations and magnetic Schr\"odinger equations, see \cite{KU201909, KU201905, LL2019, LaiL2020, LaiOhm21, LLLS201905, LLLS201903,  Lin202004}. A demonstration of the method can be found in \cite{AZ2020, AZ2018} on nonlinear Maxwell's equations, in \cite{LUY2020, LaiUhlmannZhou} on nonlinear kinetic equations, and in \cite{LLPT2020_1} on semilinear wave equations.   In \cite{LaiZhou2020}, we solved an inverse problem for the magnetic Schr\"odinger equation with nonlinearity in both magnetic and electric potentials using partial DN-map and its nonlocal fractional diffusion version \cite{LZ2023}. For the nonlinear dynamic Schr\"odinger equation considered in this paper, unique determination of time-dependent linear and nonlinear potentials from the knowledge of a source-to-solution map was discussed in \cite{LOST2022}. \\

The paper is organized as follows. In Section \ref{sec:wellposed}, we establish the well-posedness of the direct problem, the initial boundary value problem for our nonlinear time-dependent Schr\"odinger equation in a bounded domain for well chosen boundary conditions. Then we prove the local uniqueness result Theorem \ref{thm:local uniqueness} in Section \ref{sec:local} by constructing the geometrical optics (GO) solutions for the linear Sch\"odinger equation that concentrate near straight lines intersecting at a point. The higher order (multifold) linearization step is conducted via finite difference expansions in this section to derive the needed integral identity. Then we prove the stability estimate Theorem \ref{thm:stability} in Section \ref{sec:stability} where we implement a more standard type of linear GO solutions and adopt the unique continuation argument to control the boundary term due to the inaccessibility by the partial data measurement. Finally, we present the short proof of Theorem \ref{thm:unique} for a global uniqueness result by combining assumptions in the previous two theorems.

\section*{Acknowledgements}
 R.-Y. Lai is partially supported by the National Science Foundation through grant DMS-2006731.

\section{Well-posedness of the Dirichlet problem}\label{sec:wellposed}

\subsection{Notations}
Let $r$ and $s$ be two non-negative real numbers, $m$ be a non-negative integer and let $X$ be one of $\Omega$, $\p\Omega$ and $\Gamma$. We introduce the following Hilbert spaces:
\begin{itemize}
\item the space $L^2(0,T;H^s(X))$ that consists of all measurable functions $f:[0,T]\rightarrow H^s(X)$ with norm
$$
   \|f\|_{L^2(0,T;H^s(X))}:=\LC\int^T_0 \|f(t,\cdot)\|_{H^s(X)}^2\,dt \RC^{1/2}<\infty;
$$
\item the Sobolev space 
$$
H^{m}(0,T;L
^2(X)):=\{f:\,\p_t^\alpha f\in L^2(0,T;L^2(X))\quad \hbox{for }\alpha=0,1,\ldots, m\};
$$
and the interpolation  
$$
H^{r}(0,T;L^2(X))=[H^m(0,T; L^2(X)), L^2(0,T; L^2(X))]_\theta,\quad (1-\theta)m=r.
$$
\end{itemize}
We also define the Hilbert space 
$$
H^{r,s}((0,T)\times X):=H^r(0,T;L^2(X))\cap L^2(0,T;H^s(X)),
$$
whose norm is given by
$$
    \|f\|_{H^{r,s}((0,T)\times X)}:=\LC\int^T_0\|f(t,\cdot)\|^2_{H^s(X)}dt + \|f\|^2_{H^r(0,T;L^2(X))}\RC^{1/2}.
$$
For more details on these definitions, we refer to Chapter 1 and Chapter 4 in \cite{Lion}.
In particular, for integer $m\geq 1$, we define 
$$
    \mathcal{H}^{m}_0(Q):=\{f\in H^{m}(Q):\, \p_t^\alpha f|_{t=0}=0,\quad \alpha=0,\cdots, m-1\}.
$$

For $\lambda>0$ we define the subset $\mathcal{S}_\lambda(\Sigma)$ of $H^{2\kappa+\frac32,2\kappa+\frac32}(\Sigma)$ by
\begin{equation}\label{eqn:S_eps}\begin{split}
\mathcal{S}_\lambda(\Sigma):=\Big\{f\in H^{2\kappa+\frac32,2\kappa+\frac32}(\Sigma) :~ \partial_t^mf(0,\cdot)=0 \textrm{ on }\partial\Omega &\textrm{ for }\hbox{integers } m<2\kappa+{3\over 2},\\
&\textrm{ and }\quad \|f\|_{H^{2\kappa+\frac32,2\kappa+\frac32}(\Sigma)}\leq \lambda \Big\}.
\end{split}
\end{equation}

\subsection{Well-posedness}\label{sec:wellposedness}
We first show unique existence of the solution to the linear equation and, based on this, we apply the contraction mapping principle to deduce the well-posedness for the nonlinear equation.
\begin{proposition}(Well-posedness for the linear equations)\label{Prop:linear forward}
Let $2\kappa>{n+1\over 2}$ be an integer. Suppose $q\in C^\infty(Q)$.  
For any $f\in H^{2\kappa+\frac32,2\kappa+\frac32}(\Sigma)$ {satisfying $\partial_t^mf(0,\cdot)=0$ for $m<2\kappa+{3\over 2}$}, there exists a unique solution $u_{f}\in H^{2\kappa}(Q)$ to the linear system:
\begin{align}\label{IBVP}
	\left\{\begin{array}{rcll}
		\LC i \p_t +\Delta + q \RC u_{f}&=& 0 &\quad \hbox{in }Q ,\\
		u_{f} &=&f &\quad \hbox{on }\Sigma ,\\
		u_{f} &=&0 &\quad \hbox{on }\{0\}\times \Omega,\\
		\end{array} \right. 
\end{align} 
and $u_{f}$ satisfies the estimate  
\begin{align}\label{EST:u f0}
\|u_{f}\|_{H^{2\kappa}(Q)}\leq C\|f\|_{H^{2\kappa+\frac32,2\kappa+\frac32}(\Sigma)}.
\end{align}
\end{proposition}
\begin{proof}
In light of [\cite{Lion}, Chapter 4, Theorem 2.3], there exists a function $\tilde{u}\in H^{2\kappa+2,2\kappa+2}(Q)$ such that for $0\leq \alpha< 2\kappa+ {3\over 2}$,
\begin{align}\label{tilde u data}
    \p_t^\alpha\tilde{u}(0,\cdot)=0\quad \hbox{ in }\Omega,\quad \quad \tilde{u}|_{\Sigma}=f,
\end{align}
and 
$$
    \|\tilde{u}\|_{H^{2\kappa+2}(Q)}\leq C\|\tilde{u}\|_{H^{2\kappa+2, 2\kappa+2}(Q)}\leq C\|f\|_{{H}^{2\kappa+\frac32,2\kappa+\frac32}(\Sigma)}
$$
for some positive constant $C$, depending only on $\Omega$ and $T$, where the first inequality holds by noticing Proposition~2.3 in Chapter 4 in \cite{Lion}.
Let 
$$
    F:=-(i\p_t+\Delta+q)\tilde{u}.
$$
Since $\tilde u\in H^{2\kappa+2}(Q)$, we get $F\in H^{2\kappa+1,2\kappa}(Q)\subset H^{2\kappa,2\kappa}(Q)$ implying $F\in H^{2\kappa}(Q)$ by using Proposition~2.3 in Chapter 4 in \cite{Lion} again. In addition, due to \eqref{tilde u data}, $F$  
has zero initial condition up to $2\kappa$ derivative w.r.t. $t$, which makes $F\in \mathcal H^{2\kappa}_0(Q)$.
From Lemma 4 of \cite{LOST2022}, there exists a unique solution $u_*$ to the Schr\"odinger equation $(i\p_t+\Delta +q)u_*=F$ with $F|_{t=0}=0$ and $u_*|_{t=0}=u_*|_\Sigma=0$. {We denote by $\mathcal L^{-1}$ the solution operator of this inhomogeneous Dirichlet problem for the linear Schr\"odinger equation, that is, $\mathcal{L}^{-1}(F)=u_*$. In particular, we have that $\mathcal L^{-1}:~\mathcal{H}_0^{2\kappa}(Q)~\rightarrow~\mathcal{H}_0^{2\kappa}(Q)$ is a bounded linear operator.} Therefore, we obtain
$$
\|u_*\|_{H^{2\kappa}(Q)}\leq C\|F\|_{\mathcal{H}^{2\kappa}_0(Q)}\leq C\|f\|_{{H}^{2\kappa+\frac32,2\kappa+\frac32}(\Sigma)},
$$
and $u_{f} = \tilde{u}+u_*\in H^{2\kappa}(Q)$ satisfies
\[
\|u_{f}\|_{H^{2\kappa}(Q)}\leq \|\tilde u\|_{H^{2\kappa}(Q)}+\|u_*\|_{H^{2\kappa}(Q)}
\leq C\|f\|_{{H}^{2\kappa+\frac32,2\kappa+\frac32}(\Sigma)}.
\]
\end{proof}

\begin{proposition}\label{prop:forward nonlinear}(Well-posedness for the nonlinear equation)
Let $2\kappa>{n+1\over 2}$ be an integer. Suppose $q$ and $\beta$ are in $C^\infty(Q)$. For any $f\in\mathcal S_{\lambda}(\Sigma)$ (defined in \eqref{eqn:S_eps}) with $\lambda>0$ sufficiently small, there exists a unique solution $u\in H^{2\kappa}(Q)$ to the problem \eqref{IBVP nonlinear}
and it satisfies the estimate   
\begin{equation}\label{eqn:u_dep_f}
\|u\|_{H^{2\kappa}(Q)}\leq C\|f\|_{H^{2\kappa+\frac32,2\kappa+\frac32}(\Sigma)},
\end{equation}
where the constant $C>0$ is independent of $f$.
\end{proposition}
\begin{proof}
If $u$ is a solution to \eqref{IBVP nonlinear}, we set $w:=u-u_{f}$ which will solve 
\begin{equation}\label{eqn:IBVP_w}
\left\{\begin{array}{rcll}
    (i\partial_t+\Delta+q)w &=&-\beta(t,x)(w+u_{f})^2 &\quad \textrm{in }Q,\\
    w &=&0 &\quad \hbox{on }\Sigma ,\\
	w &=&0 &\quad \hbox{on }\{0\}\times \Omega,\\
\end{array}
\right.
\end{equation}
where $u_f$ is the solution to \eqref{IBVP}. Or equivalently, $w$ is the solution to
\[w-\mathcal L^{-1}\circ\mathcal Kw=0,\]
where $\mathcal Kw:=-\beta(t,x)(w+u_{f})^2$. 
For $2\kappa>\frac{n+1}2$, using the facts that $H^{2\kappa}(Q)$ is a Banach algebra and that $u_{f}\in \mathcal{H}^{2\kappa}_0(Q)$, we have that $\mathcal K:~\mathcal{H}_0^{2\kappa}~\rightarrow~\mathcal{H}_0^{2\kappa}$ is bounded.

We define for $a>0$ the subset
\[X_a(Q):=\{u\in \mathcal H^{2\kappa}_0(Q);~\|u\|_{H^{2\kappa}(Q)}\leq a\}.\]
From \eqref{EST:u f0}, we deduce
\[
 \|(\mathcal{L}^{-1}\circ \mathcal{K})w\|_{H^{2\kappa}(Q)}\leq C \|\mathcal{K}w\|_{H^{2\kappa}(Q)}\leq C\left(\|w\|^2_{H^{2\kappa}(Q)}+\|u_{f}\|_{H^{2\kappa}(Q)}^2\right)\leq C(a^2+\lambda^2)\leq a  
\]
for $w \in X_a(Q)$ and
\[\begin{split}
& \|(\mathcal{L}^{-1}\circ \mathcal{K})w_1-(\mathcal{L}^{-1}\circ \mathcal{K})w_2\|_{H^{2\kappa}(Q)}
 \leq C \|\mathcal Kw_1-\mathcal Kw_2\|_{H^{2\kappa}(Q)}\\
&\leq C\left(\|w_1\|_{H^{2\kappa}(Q)}+\|w_2\|_{H^{2\kappa}(Q)}+\|u_{f}\|_{H^{2\kappa}(Q)}\right)\|w_1-w_2\|_{H^{2\kappa}(Q)}\\
&\leq C(a+\lambda)\|w_1-w_2\|_{H^{2\kappa}(Q)}\\
&\leq K\|w_1-w_2\|_{H^{2\kappa}(Q)},\quad \hbox{ for }w_1,\,w_2\in X_a(Q)
\end{split}\]
with $K\in(0,1)$ provided that we choose $0<\lambda<a<1$ and $a$ small enough. This proves that $\mathcal L^{-1}\circ\mathcal K$ is a contraction map on $X_a(Q)$, hence there exists a fixed point $w\in X_a(Q)$ as the solution to \eqref{eqn:IBVP_w}. 
Moreover,
\[\begin{split}\|w\|_{H^{2\kappa}(Q)}=\|(\mathcal L^{-1}\circ\mathcal K) w\|_{H^{2\kappa}(Q)}&\leq C\|\mathcal Kw\|_{H^{2\kappa}(Q)}\\
&\leq C(\|w\|_{H^{2\kappa}(Q)}^2+\|u_{f}\|_{H^{2\kappa}(Q)}^2)\\
&\leq Ca\|w\|_{H^{2\kappa}(Q)}+C\lambda\|u_{f}\|_{H^{2\kappa}(Q)},
\end{split}\]
which further implies 
\[\|w\|_{H^{2\kappa}(Q)}\leq C\lambda\|u_{f}\|_{H^{2\kappa}(Q)}\]
by choosing $a$ sufficiently small.
Combined with \eqref{EST:u f0}, we eventually obtain \eqref{eqn:u_dep_f}.
\end{proof}

\section{Proof of Theorem \ref{thm:local uniqueness}}\label{sec:local}
\subsection{Geometrical optics solutions based on gaussian beam quasimodes}\label{subsec:GO_gaussian}
In this section we construct the geometrical optics solutions to the linear Schr\"odinger equation 
$$
(i\p_t + \Delta +q)u=0,
$$ in $Q$, 
having the form
\[
u(t,x) = e^{ i\rho(\Theta(x)-|\omega|^2\rho t)} a(t,x) + r(t,x)
\]
and vanishing on part of the boundary, where the leading part 
$e^{ i\rho(\Theta(x)-|\omega|^2\rho t)} a(t,x)$
follows the construction of gaussian beam approximate solutions concentrated near a straight line in direction $\omega$ as $\rho\rightarrow\infty$. For completeness, we present a detailed adaptation, to our equation, of the construction in \cite{DKLS}, which was for the operator $-\Delta_g-s^2$ on its transversal manifold $(M,g)$ and for large complex frequency $s$. The analogous construction for the wave equation can be found in \cite{KKLbook}. For other similar WKB type constructions, we refer the readers to \cite{DKLS,KianSoccorsi, LOST2022}.

Let $p$ be a point in $\Omega$ and $\omega\in\R^n$ be a nonzero direction. Denote by $\gamma_{p,\omega}$ the straight line through $p$ in direction $\omega$, parametrized by $\gamma_{p,\omega}(s)=p+s \hat\omega$ for $s\in\R$, where $\hat{\omega}:= \omega/|\omega|$. We can choose $\omega_2, \ldots, \omega_{n}\in\R^n$ such that $\mathcal{A}=\{\hat\omega,\omega_2,\ldots,\omega_{n}\}$ forms an orthonormal basis of $\R^n$. Under this basis, we identify $x\in\R^n$ by the new coordinate $z=(s,z')$ where $z':=(z_2,\ldots,z_n)$, that is, 
$$ 
    x=p+ s\hat{\omega} + z_2\omega_2+\ldots+z_{n}\omega_{n}.
$$
In particular, $\gamma_{p,\omega}(s)=(s,0,\ldots,0)$.\\

We consider the gaussian beam approximate solutions $v$ with ansatz
\begin{align}\label{AppsD}
	v(t,z)=e^{i\rho( \varphi(z)-|\omega|^2\rho t)} a(t,z;\rho),\quad \rho>0,
\end{align}
in the coordinate $(t,z)\in \R^{n+1}$. The aim is to find smooth complex functions $\varphi$ and $a$. Let the Schr\"odinger operator act on $v$ and get
\begin{align}\label{EQN:WKB}
	e^{-i\rho(\varphi(z)-|\omega|^2\rho t)}  (i\p_t+\Delta+q)v (t,z)= \rho^2 (|\omega|^2-|\nabla \varphi|^2)a +
	i \rho (2 \nabla \varphi\cdot\nabla a + a\Delta \varphi )+(i\p_t+\Delta+q) a.
\end{align}

We {\it first choose the phase function $\varphi(z)$}. 
The equation \eqref{EQN:WKB} suggests that we will choose the complex phase function $\varphi$ satisfying the eikonal equation
\begin{align*} 
	\mathcal{E}(\varphi):=|\nabla \varphi|^2 -|\omega|^2 =0 \quad \text{up to $N$-th order of $z'$ on } \gamma_{p,\omega}, 
\end{align*}
that is, 
$
    \mathcal{E}(\varphi) = O(|z'|^{N+1}).  
$
We substitute $\varphi$ of the form
\begin{align*} 
\varphi(s, z' )= \sum_{k=0}^{N} \varphi_k (s, z' ), \quad \hbox{ where }
\varphi_k (s, z' )=\sum_{|\alpha'|=k} { \varphi_{k,\alpha'}(s)\over \alpha' !}(z')^{\alpha'}.
\end{align*}
Here $\alpha$ is an $n$-dim multi-index $\alpha=(\alpha_1,\alpha')\in\mathbb Z_+^n$ with $\alpha'=(\alpha_2,\ldots,\alpha_n)$, and
 $$\varphi_0(z)=|\omega| s, \quad \varphi_1(z)=0.$$ 
We obtain
	\begin{align*} 
		|\nabla \varphi|^2 -|\omega|^2&= \underbrace{\LC2|\omega| \p_s\varphi_2 + \nabla_{z'}\varphi_2\cdot\nabla_{z'} \varphi_2\RC}_{O(|z'|^2)} + \underbrace{\LC2|\omega| \p_s\varphi_3 + 2\nabla_{z'}\varphi_2\cdot\nabla_{z'} \varphi_3\RC}_{O(|z'|^3)} \\
		&\quad +   \underbrace{\LC2|\omega| \p_s\varphi_4 + 2\nabla_{z'}\varphi_2\cdot\nabla_{z'} \varphi_4 + F_4(s,z')\RC}_{O(|z'|^4)} +\cdots +O(|z'|^{N+1}),
	\end{align*} 
where $F_j(s,z')$ is a $j^{th}$ order homogeneous polynomial in $z'$ depending only on $\varphi_2,\ldots,\varphi_{j-1}$. 
Next we look for $\varphi_2$ such that the first $O(|z'|^2)$ term vanish. Writing 
$$
    \varphi_2(s,z')={1\over 2}H(s)z'\cdot z', 
$$
where $H(s)=(H_{ij}(s))_{2\leq i,j\leq n}$ is a smooth complex symmetric matrix.
Then $H$ satisfy the matrix Riccati equation
\begin{align}\label{EQN:Riccati}
    |\omega|{d\over ds} H(s)+H^2(s)=0.
\end{align}
Imposing an initial condition $H(0)=H_0$, where $H_0$ is a complex symmetric matrix with positive definite imaginary part $\textrm{Im}H_0$, by [\cite{KKLbook} Lemma~2.56], there exists a unique smooth complex symmetric solution $H(s)$ to \eqref{EQN:Riccati} with positive definite $\textrm{Im}H(s)$ for all $s\in\R$.

For $|\alpha|\geq 3$, in order to make the $O(|z'|^3), \ldots, O(|z'|^N)$ terms vanish, 
one derives first order ODE's for the Taylor coefficients $\varphi_{k,\alpha'}$. By imposing well-chosen initial conditions at $s=0$, we may find all the $\varphi_j$, $j=3, \ldots, N$. \\

{\it Next we construct the amplitude function $a(t,z; \rho)$.}
Let $\chi_{\eta}\in C^\infty_c(\R^{n-1})$ be a smooth function with $\chi_\eta=1$ for $|z'|\leq {\eta\over 2}$ and $\chi_\eta=0$ for $|z'|\geq \eta$.  
Let $\iota \in C_0^{\infty}(0,T)$ be a smooth cut-off function of the time variable.  
We make the ansatz for the amplitude as
$$  
    a(t,s,z';\rho)=\sum_{j=0}^{N}\rho^{-j}a_j(t,s,z') \chi_{\eta}(z')=(a_0 + \rho^{-1}a_1  +\cdots + \rho^{-N}a_{N}  )\chi_{\eta}(z').
$$
From \eqref{EQN:WKB}, we should determine $a_j$ from
\begin{align}\label{TREs}
\begin{array}{ll}
2  \nabla \varphi \cdot \nabla a_0 + a_0\Delta\varphi =0   &~~\text{up to $N$-th order of $z'$ on } \gamma_{p,\omega},\\
2  \nabla \varphi \cdot \nabla  a_1 + a_1\Delta \varphi =i (i\p_t+\Delta+q)a_0   &~\text{up to $N$-th order of $z'$ on } \gamma_{p,\omega},  \\
  \hskip3cm\vdots&\\
 2  \nabla \varphi \cdot \nabla a_{N} + a_{N} \Delta\varphi=i (i\p_t+\Delta+q)a_{N-1}  &~\text{up to $N$-th order of $z'$ on } \gamma_{p,\omega}.
\end{array} 
\end{align}
so that the terms of $O(\rho^{-k})$ ($k=0,\ldots, N$) vanish up to $N$-th order of $z'$ on $\gamma_{p,\omega}$.
Therefore, we write $a_0$ to have the form
$$
    a_0(t,s,z')=\sum_{k=0}^{N}a^{ k}_0 (s,z') \iota(t),\quad \hbox{where}\quad a^{ k}_0 ( s,z') =\sum_{|\alpha'|=k} {a_0^{k,\alpha'}(s) \over \alpha' !}(z')^{\alpha'} .
$$
Here $a^k_0$ is a $k^{th}$ order homogeneous polynomial in $z'$. 
The first equation in \eqref{TREs} becomes
\begin{align}\label{EQN:a0}
 2 \nabla \varphi \cdot \nabla a_0  + a_0\Delta\varphi 
&= \iota(t) \LC 2 |\omega| \p_s a_0^0 + a_0^0 \Delta_{z'} \varphi_2\RC \notag\\
&\quad+ \iota(t)\LC 2 |\omega| \p_sa^1_0 + 2  \nabla_{z'}\varphi_2\cdot \nabla_{z'}a^1_0 + a^1_0\Delta_{z'}\varphi_2 +a^0_0\Delta_{z'}\varphi_3\RC + \cdots+O(|z'|^{N+1}).
\end{align}
 Note that $\Delta_{z'} \varphi_2=tr(H(s))$. 
In order to let the first bracket vanish, we solve $2 |\omega| \p_s a_0^0(s) + tr(H(s))a_0^0 (s)=0$ with a given initial condition $a_0^0(0)=c_0$ for some constant $c_0$. For later purpose, we choose $c_0=1$ to get
$$
a_0^0(s)= e^{-\frac{1}{2|\omega|} \int_{0}^s tr(H(t))dt}. 
$$
 Similarly, the coefficients of $a^1_0, \ldots, a^N_0$ can be determined for the other brackets in \eqref{EQN:a0} to vanish.  
Lastly, we can also construct $a_2 ,\ldots, a_{N}$, which have similar forms as $a_0$, in a similar way.
Here we note that $a_0^{k,\alpha'}$ is smooth which further implies that $a(t,z;\rho)$ is smooth.

So far we have constructed 
a gaussian beam $v(t,z)$ localized near $\{(z_1,0,\ldots,0), z_1\in\R\}$ of the form \eqref{AppsD}
with
$$ \varphi(s,z')=|\omega|s+\frac12H(s)z'\cdot z' + O(|z'|^3),\quad 
a(t,s,z')=\chi_{\eta}(z')(a_0+\rho^{-1}a_1 +
\cdots+\rho^{-N}a_N) 
$$
with positive definite $\mathrm{Im}H (s)$.

{It is easy to verify that by translation and rotation $\Psi(x)=z$, the function defined by $v(t,\Psi(x))$ with $a(t,\Psi(x))$, still denoted by $v(t,x)$ and $a(t,x)$ respectively, is indeed the gaussian beam localized near the line $\gamma_{p,\omega}$ and satisfy 
\begin{align} \label{EST:Calulation}
	&(i\p_t+\Delta_x+q(t,x))v(t,x)
	 = (i\p_t + \Delta_z+q) v(t,z) \notag \\  
	 =\,& e^{ i\rho(\varphi(z)-|\omega|^2\rho t)}
  \bigg( 
  \chi_{\eta}(z')
 \LC O(|z'|^{N+1}) \rho^2 +O(|z'|^{N+1}) \rho+ (i\p_t+\Delta +q) a_{N}\rho^{-N}\RC
 +   \rho\widehat{\chi}_{\eta}(z') \vartheta\bigg),
\end{align}
where $q(t,x)$ here is the above $q(t,z)$ with $z=\Psi(x)$ (We do not distinguish the names of the functions, e.g. $q(t,x)$ and $q(t,z)$, but only indicate the difference due to transformation by notations of variables $(t,x)$ and $(t,z)$) and $\widehat{\chi}_{\eta}(z')$ is a smooth function with $\widehat{\chi}_\eta=0$ for $|z'|<{\eta\over 2}$ and $|z'|\geq  \eta$, and  $\vartheta$ vanishes near the geodesic $\gamma_{p,\omega}.$ 
This last term accounts for those derivatives landing on $\chi_\eta$. 

More specifically, we have
\begin{align}\label{approrgs}
	\begin{split}
		v(t,x)= e^{ i\rho(\Theta(x)-|\omega|^2\rho t)} a(t,x),
	\end{split}
\end{align}
where the phase function is explicitly given by 
$$
    \Theta(x)=\varphi(\Psi(x))= \omega  \cdot (x-p) +\frac12  \mathcal{H}(x) (x-p)\cdot(x-p)+O(\textrm{dist}(x,\gamma_{p,\omega})^3),
$$
where $\mathcal{H}(x)$ is an $n\times n$ matrix, defined by
\[\mathcal H(x)=D\Psi(x)\left(
\begin{array}{cc}
	0 & 0   \\
	0 & H((x-p)\cdot\widehat\omega)  \\
\end{array}
\right)(D\Psi(x))^T,\]
and the notation $\text{dist}(x,\gamma_{p,\omega})$ represents the distance between the point $x$ and the line $\gamma_{p,\omega}$.
Moreover, based on the properties of $H$, that is, $\mathrm{Im}{H}(s)$ is positive definite,
combined with the fact that $D\Psi$ is a {unitary} matrix, 
we have 
that there exists a constant $c_0>0$ such that
\begin{equation}\label{eqn:H_pos_def}
\frac12  \mathrm{Im}\mathcal{H}(x) (x-p)\cdot(x-p)\geq c_0(\textrm{dist}(x,\gamma_{p,\omega})^2)\qquad \textrm{ for all } x.
\end{equation}
 
To summarize, we obtain
\begin{proposition}\label{prop:WKB v} Let $q\in C^\infty(Q)$ and $\gamma_{p,\omega}$ be a straight line through a point $p\in \Omega$ in direction $\omega\in\R^n$. For any $N>0$ and $\eta>0$, there exists a family of approximate solutions $\{v_\rho\in C^\infty(Q),~\rho>1\}$, supported in $(0,T)\times N_\eta(\gamma_{p,\omega})$ where $N_\eta(\gamma_{p,\omega})$ is an $\eta$-neighborhood of $\gamma_{p,\omega}$, such that 
\begin{align}\label{EST:WKB v}
	\|(i\p_t+\Delta_x+q)v_\rho\|_{\HLT} \leq C\rho^{-\frac{N+1}{2}-\frac{n-1}{4}+4},
\end{align}
and, for integer $m\geq 0$,
\begin{align}\label{EST:WKB v Hm}
	\|(i\p_t+\Delta_x+q) v_\rho\|_{H^m(Q)} \leq C  \rho^{-\frac{N+1}{2}-\frac{n-1}{4}+2m+2},
\end{align}
where $C$ is a positive constant independent of $\rho$.
\end{proposition}
}
\begin{proof}
	Take $v_\rho$ as in \eqref{approrgs}. It remains to show \eqref{EST:WKB v} and \eqref{EST:WKB v Hm}. To begin with, since $\mathrm{Im}(H(s))$ is positive definite, there exists $c_1>0$ so that $\mathrm{Im}(H(s))z'\cdot z'\geq c_1|z'|^2$.  
	Therefore, for $\eta<1$ sufficiently small, in the neighborhood $\{|z'|<\eta\}$ one has 
	$$
	    | e^{ i\rho(\varphi(s,z')-|\omega|^2\rho t)}| 
	    \leq   e^{-\frac14 c_1 \rho |z'|^2 }.
	$$
	The equation \eqref{EST:Calulation} implies
	 \begin{align*}
	 	 | (i\p_t+\Delta_x+q) v_\rho| \leq C e^{-\frac14 c_1 \rho |z'|^2 } \LC |z'|^{N+1} \rho^2  \chi_{\eta}(z')+ \rho^{-N}\chi_{\eta}(z') +  
         \rho\widehat{\chi}_{\eta}(z') \vartheta \RC,
    \end{align*}
    \begin{align*}
	   	| \p_t(i\p_t+\Delta_x+q) v_\rho| \leq C e^{-\frac14 c_1 \rho |z'|^2 }
	   	\LC |z'|^{N+1} \rho^4  \chi_{\eta}(z')+  
	   	\rho^{2-N}   \chi_{\eta}(z')+  
	   	\rho^3\widehat{\chi}_{\eta}(z') \vartheta \RC.
    \end{align*}
	Hence it follows that
    \begin{align}\label{integral 1}
	    &\quad \|   (i\p_t+\Delta_x+q) v_\rho\|^2_{\HLT} \notag\\
	  	&\leq  C \rho^8 \int_0^T \| e^{-\frac14 c_1 \rho |z'|^2 }  |z'|^{N+1}  \chi_{\eta}(z')\|^2_{L^2(\Om)} dt+
         C \rho^{-2N+4}\int_0^T \| e^{-\frac14 c_1 \rho |z'|^2 }     \chi_{\eta}(z')\|^2_{L^2(\Om)} dt\notag  \\
         &\quad + C \rho^6 \int_0^T \| e^{-\frac14 c_1 \rho |z'|^2 }  \widehat{ \chi}_{\eta}(z') \vartheta \|^2_{L^2(\Om)} dt =:J_1+J_2+J_3. 
	\end{align}

	Now by changing of variable $z'=\rho^{-\frac12}y$ and applying integration by parts, we obtain 
	\begin{align}\label{integral 2}
         \quad  J_1  
         \notag
         &\leq C \rho^8 \int_{|z'|\leq \eta} e^{-\frac12 c_1 \rho |z'|^2 } |z'|^{2N+2} dz'\notag\\
		 &\leq  C \rho^{-N-1-\frac{n-1}{2}+8}\int_{\R^{n-1}} e^{-\frac12 c_1 |y|^2}|y|^{2N+2} dy \notag\\
		 &\leq C \rho^{-N-1-\frac{n-1}{2}+8},
	\end{align}
    where the constant $C>0$ is independent of $\rho$.
    Likewise, we can also deduce
    \begin{align} \label{integral 4}
         J_2 \leq C\rho^{-2N-\frac{n-1}{2}+4},
    \end{align}
     which is controlled by \eqref{integral 2} provided $\rho$ is sufficiently large. Moreover, since $\widehat{\chi}_\eta$ is supported in ${\eta\over 2}\leq |z'|\leq \eta$, by performing the change of variable $z'=\rho^{-\frac12}y$ again, we derive
  	\begin{align}\label{integral 3}
        J_3  
       &\leq   C\rho^6\int_{\frac{\eta}{2} \leq |z'|\leq \eta} e^{-\frac12 c_1 \rho |z'|^2 }  dz' \notag\\
       &\leq  C   \rho^{ -\frac{n-1}{2}+6} \int_{\frac{\eta}{2}  \rho^{\frac12} \leq |y|\leq \eta \rho^{\frac12}} e^{-\frac12 c_1 |y|^2} dy  \notag\\ 
	   &\leq   C  \rho^{-\frac{n-1}{2}+6}  e^{-\frac18 c_1 \eta^2\rho} (\eta\rho^{\frac12})^{ n-1 }\notag\\ 
	   &\leq   C \eta^{  n-1 }  e^{-\frac18 c_1 \eta^2\rho}\rho^{ 6}, 
 	\end{align}
     which decays exponentially in $\rho$ (for a fixed $\eta$) and is also controlled by \eqref{integral 2} provided $\rho$ is sufficiently large. 
 	 Therefore, \eqref{EST:WKB v} holds by combining \eqref{integral 1}, \eqref{integral 2}, \eqref{integral 4} and \eqref{integral 3}. 

    Similarly, we have the following  higher regularity estimate 
    \begin{align*}
	     &\quad \|   (i\p_t+\Delta_x+q) v\|^2_{H^m(Q)} \\
         &\leq  C  \rho^{4m+4}  \int_0^T \| e^{-\frac14 c_1 \rho |z'|^2 }  |z'|^{N+1} \chi_{\eta}(z')\|^2_{L^2(\Om)} dt +  C \rho^{-2N+4m}\int_0^T \| e^{-\frac14 c_1 \rho |z'|^2 }     \chi_{\eta}(z')\|^2_{L^2(\Om)} dt\notag  \\
         &\quad + C \rho^{4m+2}  \int_0^T \| e^{-\frac14 c_1 \rho |z'|^2 }   \widehat{ \chi}_{\eta}(z') \vartheta \|^2_{L^2(\Om)} dt \\
         &\leq C  \rho^{-N-1-\frac{n-1}{2}+4m+4},
    \end{align*}
    provided $\rho$ is sufficiently large. This completes the proof of \eqref{EST:WKB v Hm}. 
\end{proof}

With Proposition~\ref{prop:WKB v}, we can construct the geometrical optics solutions now.
\begin{proposition} \label{prop:WKB solutions}
Let $m>0$ be an even integer and $q\in C^\infty(Q)$. Given $p\in\Omega$ and $\omega\in\R^n$, suppose that the straight line $\gamma_{p,\omega}$ through $p$ in direction $\omega$ satisfies $(\gamma_{p,\omega}\cap \p\Omega)\subset\Gamma$. Then there exists $\rho_0>1$ such that when $\rho>\rho_0$, the Schr\"odinger equation $(i \p_t +\Delta +q )u=0$ admits a solution $u\in H^m(Q)$ of the form 
\begin{align*}
    u(t,x) = e^{ i\rho(\Theta(x)-|\omega|^2\rho t)} a(t,x) + r(t,x)
\end{align*}
with boundary value $\text{supp}(u|_{(0,T)\times\p\Omega})\subset \Sigma^\sharp$ and initial data $u|_{t=0}=0$ in $\Omega$ (or the final condition $u|_{t=T}=0$ in $\Omega$). Here $\Theta(x)$ and $a(t,x)$ are as in \eqref{approrgs} and satisfy Proposition \ref{prop:WKB v} and
the remainder $r$ satisfies the following estimates:
\begin{align}\label{remdS 0}
    \|r\|_{H^{m}(Q)}\leq C \rho^{-\frac{N+1}{2}-\frac{n-1}{4}+2m+2}
\end{align}
and
\begin{align*} 
    \|r\|_{\CHT}+\|r\|_{\CLT}\leq C \rho^{-\frac{N+1}{2}-\frac{n-1}{4}+4}.  
\end{align*}

\end{proposition}
\begin{proof}
 
{
We can choose $\eta>0$ small enough such that $(N_\eta(\gamma_{p,\omega})\cap\partial\Omega)\subset\Gamma$. By the previous Proposition \ref{prop:WKB v}, for $\rho>\rho_0$, we obtain $\Theta(x)$ and $a(t,x)$ correspondingly.}
By Proposition~3 and Lemma~4 in \cite{LOST2022}, we obtain the existence of the solution $r\in H^{m}(Q)$ to
\begin{align*} 
	\left\{\begin{array}{rcll}
    (i\p_t  +\Delta +q)r  &=&-(i\p_t+\Delta+q)v   &\quad \hbox{in }Q,\\
    r &=& 0 &\quad \hbox{on } \Sigma,\\
    r &=&0 &\quad \hbox{on } \{0\}\times \Om,
	\end{array} \right.
\end{align*}
and the estimate
$$
    \|r\|_{H^{m}(Q)}\leq C \|(i\p_t+\Delta+q)v\|_{H^{m}(Q)}\leq C
    \rho^{-\frac{N+1}{2}-\frac{n-1}{4}+2m+2}. 
$$
Here the last inequality follows from Proposition \ref{prop:WKB v}. 
Also, with \eqref{EST:WKB v}, [\cite{KianSoccorsi}, Lemma 2.3] suggests \begin{align*} 
    \|r\|_{\CHT}+\|r\|_{\CLT}\leq C \|(i\p_t+\Delta+q)v \|_{\HLT}\leq C \rho^{-\frac{N+1}{2}-\frac{n-1}{4}+4}.
\end{align*}
 
\end{proof}

\subsection{Finite difference}
We introduce the multivariate finite differences, which are approximations to the derivative. We define the second-order mixed finite difference operator $D^2$ about the zero solution as follows:
\begin{align*}
    D^2 u_{\varepsilon_1f_2+\varepsilon_2f_2} := {1\over \varepsilon_1\varepsilon_2} (u_{\varepsilon_1f_1+\varepsilon_2f_2} - u_{\varepsilon_1f_1} - u_{\varepsilon_2 f_2}).
\end{align*}
Note that when $\varepsilon_1=\varepsilon_2=0$,  
$u_{\varepsilon_1f_2+\varepsilon_2f_2}=0$. 
We refer the interested readers to \cite{LOST2022} for the definitions of higher order finite difference operators. For the purpose of our paper, we only need $D^2$.
To simplify the notation, we denote $u_{\varepsilon_1f_2+\varepsilon_2f_2}$ by $u_{\varepsilon f}$ and define $|\varepsilon|:=|\varepsilon_1|+|\varepsilon_2|$. Then we have the following second order expansion. 

\begin{proposition}\label{prop:decomposition}
Let $2\kappa>{n+1\over 2}$ be an integer and $f_j\in H^{2\kappa+\frac32,2\kappa+\frac32}(\Sigma)$ for $j=1,2$. For $|\varepsilon|:=|\varepsilon_1|+|\varepsilon_2|$ small enough, there exists a unique solution $u_{\varepsilon f}\in H^{2\kappa}(Q)$ to the problem 
\begin{align*} 
	\left\{\begin{array}{rcll}
		(i \p_t +\Delta +q)u_{\varepsilon f} + \beta  u_{\varepsilon f}^2 &=& 0 &\quad \hbox{in }Q,\\
		u_{\varepsilon f}&=&\varepsilon_1f_1+\varepsilon_2f_2 &\quad \hbox{on }\Sigma,\\
		u_{\varepsilon f} &=&0 &\quad \hbox{on }\{0\}\times \Omega.
	\end{array} \right. 
\end{align*} 
In particular, it admits the following expression:
\begin{align*} 
    u_{\varepsilon f}=\varepsilon_1U_1+\varepsilon_2U_2+{1\over 2}\LC\varepsilon_1^2W_{(2,0)}+  \varepsilon_2^2W_{(0,2)}+2\varepsilon_1\varepsilon_2 W_{(1,1)}\RC + \mathcal{R},
\end{align*}
where for $j=1,2$, $U_{j}\in H^{2\kappa}(Q)$ satisfies the linear equation: 
\begin{align}\label{IBVP linear U}
	\left\{\begin{array}{rcll}
		(i \p_t +\Delta +q)U_{j} &=& 0 &\quad \hbox{in }Q,\\
		U_{j}&=& f_j &\quad \hbox{on }\Sigma,\\
		U_{j} &=&0 &\quad \hbox{on }\{0\}\times \Omega,\\
	\end{array} \right. 
\end{align} 
and for $k_j\in \{0,1,2\}$ satisfying $k_1+k_2=2$, $W_{(k_1,k_2)}\in H^{2\kappa}(Q)$ is the solution to 
\begin{align}\label{IBVP linear W}
	\left\{\begin{array}{rcll}
		(i \p_t + \Delta + q)W_{(k_1,k_2)}&=& - 2\beta  U_1^{k_1}U_2^{k_2}&\quad \hbox{in }Q,\\
		W_{(k_1,k_2)} &=& 0&\quad \hbox{on }\Sigma,\\
		W_{(k_1,k_2)} &=&0 &\quad \hbox{on }\{0\}\times \Omega.
	\end{array} \right. 
\end{align} 
Moreover, the remainder term $\mathcal{R} \in H^{2\kappa}(Q)$ satisfies  
\begin{align}\label{EST:R}
    \|\mathcal{R}\|_{H^{2\kappa}(Q)} \leq C \|\varepsilon_1f_1+\varepsilon_2f_2 \|_{H^{2\kappa+\frac32,2\kappa+\frac32}(\Sigma)}^3.
\end{align}
\end{proposition}
 
\begin{proof}
{
The existence of $u_{\varepsilon f}\in {H}^{2\kappa}(Q)$ is given by Proposition~\ref{prop:forward nonlinear} when $|\varepsilon|:=|\varepsilon_1|+|\varepsilon_2|$ sufficiently small such that $\varepsilon_1f_1+\varepsilon_2f_2\in \mathcal S_\lambda(\Sigma)$.
Also, equations \eqref{IBVP linear U} and \eqref{IBVP linear W} are both well-posed in $H^{2\kappa}(Q)$, for example by Proposition 4 in \cite{LOST2022}, for $2\kappa$ as in the assumption (${H}^{2\kappa}(Q)$ is a Banach algebra). 
}
We denote
$$
    \tilde u:=u_{\varepsilon f} -  (\varepsilon_1 U_1+\varepsilon_2 U_2).
$$
Then it solves
\begin{align*} 
	\left\{\begin{array}{rcll}
		(i \p_t  + \Delta +q)\tilde u &=&- \beta u_{\varepsilon f}^2 &\quad \hbox{in }Q,\\
	    \tilde u&=& 0&\quad \hbox{on }\Sigma,\\
		\tilde u &=&0 &\quad \hbox{on }\{0\}\times \Omega,\\
	\end{array} \right. 
\end{align*} 
Applying  Lemma~4 in \cite{LOST2022} and \eqref{EST:u f0} gives that
\begin{align}\label{EST:tilde u} 
    \|\tilde{u}\|_{H^{2\kappa}(Q)}
    \leq C \| \beta u_{\varepsilon f}^2 \|_{{H}^{2\kappa}(Q)}  
    \leq C \| u_{\varepsilon f} \|^2_{{H}^{2\kappa}(Q)} 
    \leq C \|\varepsilon_1f_1+\varepsilon_2f_2 \|_{H^{2\kappa+\frac32,2\kappa+\frac32}(\Sigma)}^2.
\end{align}

From \eqref{IBVP linear U} and \eqref{IBVP linear W}, the remainder $\mathcal{R}$ satisfies
 \begin{align*} 
	\left\{\begin{array}{rcll}
		(i \p_t  + \Delta +q)\mathcal{R}  &=&- \beta u_{\varepsilon f}^2 + 
  \beta(\varepsilon_1 U_1 + \varepsilon_2 U_2)^2&\quad \hbox{in }Q,\\
	    \mathcal{R}&=& 0&\quad \hbox{on }\Sigma,\\
		\mathcal{R} &=&0 &\quad \hbox{on }\{0\}\times \Omega .\\
	\end{array} \right. 
\end{align*} 
Then we have that $\mathcal{R}\in {H}^{2\kappa}(Q)$ exists and satisfies  
\begin{align*} 
    \|\mathcal{R}\|_{H^{2\kappa}(Q)}
    &\leq C \|- \beta u_{\varepsilon f}^2 + 
    \beta(\varepsilon_1 U_1 + \varepsilon_2 U_2)^2\|_{{H}^{2\kappa}(Q)} \\
    & \leq C\| \tilde u \|_{{H}^{2\kappa}(Q)}\| u_{\varepsilon f} +
      (\varepsilon_1 U_1+\varepsilon_2 U_2) \|_{{H}^{2\kappa}(Q)}\\
    &\leq C \|\varepsilon_1f_1+\varepsilon_2f_2 \|_{H^{2\kappa+{3\over 2},2\kappa+{3\over 2}}(\Sigma)}^2\|\varepsilon_1f_1+\varepsilon_2f_2 \|_{H^{2\kappa+{3\over 2},2\kappa+{3\over 2}}(\Sigma)}\\
    &\leq C \|\varepsilon_1f_1+\varepsilon_2f_2 \|_{H^{2\kappa+{3\over 2},2\kappa+{3\over 2}}(\Sigma)}^3 
\end{align*}
by using the fact that $H^{2\kappa}(Q)$ is a Banach algebra, the equations \eqref{EST:tilde u}, \eqref{EST:u f0} and the well-posedness of \eqref{IBVP linear U}.
\end{proof}
 
\begin{remark}\label{rk:W_epsorder}
Based on Proposition~\ref{prop:decomposition}, when one of $\varepsilon_1$ and $\varepsilon_2$ is zero, we have
\begin{align*} 
     u_{\varepsilon_1 f_1}=\varepsilon_1U_1+ {1\over 2} \varepsilon_1^2W_{(2,0)}  + \mathcal{R}^{(1)},\quad u_{\varepsilon_2 f_2}=\varepsilon_2 U_2+{1\over 2} \varepsilon_2^2W_{(0,2)}  + \mathcal{R}^{(2)},
\end{align*}
where $\mathcal{R}^{(j)}$ is the remainder term of order $O(\varepsilon_j^3)$ for $j=1,\, 2$.
We can rewrite $u_{\varepsilon f}$ as
\begin{equation}\label{eqn:tilde_R}
     u_{\varepsilon f}=u_{\varepsilon_1 f_1}+u_{\varepsilon_2 f_2}+ \varepsilon_1\varepsilon_2 W_{(1,1)} + \widetilde{\mathcal{R}},
\end{equation}
where $\widetilde{\mathcal{R}}:=\mathcal{R}-\mathcal{R}^{(1)}-\mathcal{R}^{(2)}$.
Moreover, we have
\begin{align*}
     W_{(1,1)} = D^2 u_{\varepsilon_1f_1+\varepsilon_2f_2} -{1\over \varepsilon_1\varepsilon_2}\widetilde{\mathcal{R}} 
\end{align*}
and also the Neumann data
\begin{align*} 
    \p_{\nu} W_{(1,1)}|_{\Sigma^\sharp} = \frac{1}{\varepsilon_1 \varepsilon_2 }\LC\Lambda_{q,\beta} (\varepsilon_1 f_1+\varepsilon_2f_2)- \Lambda_{q,\beta}(\varepsilon_1 f_1) -   \Lambda_{q,\beta}(\varepsilon_2 f_2)\RC - \frac{1}{\varepsilon_1 \varepsilon_2 }\p_{\nu}\widetilde{\mathcal{R}}|_{\Sigma^\sharp} .
\end{align*}
Through the rest of the paper, we only need to assume $|\varepsilon_1|\sim |\varepsilon_2|\sim|\varepsilon|$, in which case we have $\widetilde{\mathcal{R}}=o(\varepsilon_1 \varepsilon_2)$. 
In fact, from \eqref{EST:R} we have
\[
\|\widetilde R\|_{H^{2\kappa}(Q)} 
\leq C(\varepsilon_1+\varepsilon_2)^3\left(\|f_1\|^3_{H^{2\kappa+\frac32, 2\kappa+\frac32}(\Sigma)}+\|f_2\|^3_{H^{2\kappa+\frac32, 2\kappa+\frac32}(\Sigma)}\right)^3.
\]

In the case that $\varepsilon_1$ and $\varepsilon_2$ are of different scales such as $|\varepsilon_2|\sim|\varepsilon_1|^k$ for some positive $k>1$ (or vice versa), more terms can be taken in the expansions of $u_{\varepsilon f}$, $u_{\varepsilon_1f_1}$ and $u_{\varepsilon_2f_2}$ to eventually  verify that $\widetilde{\mathcal R}$ has the norm of order $o(\varepsilon_1\varepsilon_2)$.

Since $W_{(k_1,k_2)}$ is independent of $\varepsilon_1$ and $\varepsilon_2$, this implies 
 \begin{align}\label{eqn:W_indep_eps}
 W_{(1,1)} =  \lim_{\varepsilon_1,\varepsilon_2\rightarrow0}D^2 u_{\varepsilon_1f_1+\varepsilon_2f_2},\quad
    \p_{\nu} W_{(1,1)}|_{\Sigma^\sharp} = \lim_{\varepsilon_1,\varepsilon_2\rightarrow0}\frac{1}{\varepsilon_1 \varepsilon_2 }\LC\Lambda_{q,\beta} (\varepsilon f)- \Lambda_{q,\beta}(\varepsilon_1 f_1) -   \Lambda_{q,\beta}(\varepsilon_2 f_2)\RC. 
\end{align}
in proper norms. For example, in $L^2(\Sigma^\sharp)$, we can derive
\begin{align}
\label{eqn:W11_L2bdry}
    &\quad \left\|\p_\nu W_{(1,1)}|_{\Sigma^\sharp}-\frac{1}{\varepsilon_1\varepsilon_2}\left(\Lambda_{q,\beta} (\varepsilon f)- \Lambda_{q,\beta}(\varepsilon_1 f_1) -   \Lambda_{q,\beta}(\varepsilon_2 f_2)\right)\right\|_{L^2(\Sigma^\sharp)}\nonumber\\
    &= \frac{1}{\varepsilon_1\varepsilon_2}\|\p_\nu \widetilde{\mathcal{R}}\|_{L^2(\Sigma^\sharp)}\leq \frac{1}{\varepsilon_1\varepsilon_2}\|\p_\nu \widetilde{\mathcal{R}}\|_{H^{2\kappa-{3\over 2},2\kappa-{3\over 2}}(\Sigma^\sharp)}
     \leq C\frac{1}{\varepsilon_1\varepsilon_2}\|\widetilde{\mathcal{R}}\|_{H^{2\kappa ,2\kappa}(Q)}
     \leq C\frac{1}{\varepsilon_1\varepsilon_2}\|\widetilde{\mathcal{R}}\|_{H^{2\kappa }(Q)}\nonumber\\
     &\leq C\frac{(\varepsilon_1+\varepsilon_2)^3}{\varepsilon_1\varepsilon_2}\left(\|f_1\|_{H^{2\kappa+\frac32, 2\kappa+\frac32}(\Sigma)}+\|f_2\|_{H^{2\kappa+\frac32, 2\kappa+\frac32}(\Sigma)}\right)^3.
     \end{align}
\end{remark}

\subsection{An integral identity}
Let $u_{\ell,\varepsilon f}$ ($\ell=1,2$) be the small unique solution to the initial boundary value problem for the Schr\"odinger equation:
\begin{align*} 
	\left\{\begin{array}{rcll}
		(i \p_t +\Delta +q)u_{\ell,\varepsilon f} + \beta_\ell u_{\ell,\varepsilon f}^2 &=& 0 &\quad \hbox{in }Q,\\
		u_{\ell,\varepsilon f} &=&\varepsilon_1f_1+\varepsilon_2f_2 &\quad \hbox{on }\Sigma,\\
		u_{\ell,\varepsilon f} &=&0 &\quad \hbox{on }\{0\}\times \Omega\\
	\end{array} \right. 
\end{align*} 
with $\text{supp}(f_j)\subset (0,T)\times\Gamma$ for $j=1,2$. 
For $|\varepsilon|:=|\varepsilon_1|+|\varepsilon_2|$ small enough, they admit the expansion
\begin{align*} 
    u_{\ell,\varepsilon f}=\varepsilon_1U_{\ell,1}+\varepsilon_2U_{\ell,2}+{1\over 2}\LC \varepsilon_1^2W_{\ell,(2,0)}+\varepsilon_2^2W_{\ell,(0,2)}+2\varepsilon_1\varepsilon_2 W_{\ell,(1,1)}\RC + \mathcal{R_\ell},
\end{align*}
where $U_{\ell,j}$, $W_{\ell,(k_1,k_2)}$ and $\mathcal R_\ell$ are as in Proposition \ref{prop:decomposition}. Since the linearized equations for both $\ell$ are the same with the same boundary data $f_j$, we have 
\[U_{1,j}=U_{2,j}, \qquad j=1,2,\]
denoted by $U_j$ for the rest of the paper.

In addition, let $U_0$ be the solution of the adjoint problem:
\begin{align}\label{IBVP linear U0 WKB}
	\left\{\begin{array}{rcll}
		(i \p_t +\Delta + q)U_0 &=& 0 &\quad \hbox{in }Q,\\
		U_0&=& f_0 &\quad \hbox{on }\Sigma,\\
		U_0 &=&0 &\quad \hbox{on }\{T\}\times \Omega\\
	\end{array} \right. 
\end{align} 
with $\text{supp}(f_0)\subset (0,T)\times\Gamma$.

\begin{lemma}\label{lemma:ID}
    Let $q, \,\beta_\ell\in C^\infty(Q)$ ($\ell=1,2$) and $\beta:=\beta_1-\beta_2$. Suppose that $$\Lambda_{q,\beta_1}(f)=\Lambda_{q,\beta_2}(f)$$ for all $f\in \mathcal{S}_\lambda(\Sigma)$ with $\text{supp}(f)\subset \Sigma^\sharp$. Then 
    \begin{align}\label{ID:local uniqueness}
    	\int_Q \beta U_1U_2\overline{U}_0\,dxdt =0.
    \end{align}
\end{lemma} 
\begin{proof}
	We denote 
	$$
	W:=W_{2,(1,1)}-W_{1,(1,1)}.
	$$
By \eqref{eqn:W_indep_eps}, we have $\p_\nu W|_{\Sigma^\sharp}=0$.
After multiplying the equation in \eqref{IBVP linear W} by $\overline{U}_0$, subtracting and integrating over $Q$, we have
\begin{align*}
	\int_Q 2 \beta U_1U_2 \overline{U}_0 \,dxdt
	=\int_\Sigma (\overline{U}_0\p_\nu W - W\p_\nu\overline{U}_0)\,d\sigma(x) dt=0 
\end{align*}
due to that $U_0(T,\cdot)=W(0,\cdot)=0$, $W|_\Sigma=\p_\nu W|_{\Sigma^\sharp}=0$ and $U_0|_{\Sigma}$ has the support in $\Sigma^\sharp$.

\end{proof}

\subsection{Proof of Theorem \ref{thm:local uniqueness}}
We will show that the coefficient $\beta(t,x)$ can be recovered uniquely for all the points in $(0,T)\times\Omega_\Gamma$.

\begin{proof}[Proof of Theorem~\ref{thm:local uniqueness}]
For each $p\in \Omega_\Gamma$, choose $\omega_1,\, \omega_2\in \mathbb S^{n-1}$ satisfying the condition in the description of $\Omega_\Gamma$ in \eqref{DEF:Omega Gamma}. Set $\omega_0:=\omega_1+\omega_2$. 
Based on Proposition~\ref{prop:WKB solutions}, 
we can find geometrical optics solutions 
$U_j=v_j+r_j$, $j=1,2$ for the problem \eqref{IBVP linear U} and $U_0=v_0+r_0$ for its adjoint problem \eqref{IBVP linear U0 WKB} associated to three lines $\gamma_{p,\omega_1}, \gamma_{p,\omega_2}$ and $\gamma_{p,\omega_0}$ respectively. More specifically, we have 
$$
	v_j(t,x)= e^{ i\rho(\Theta_j(x)-|\omega_j|^2\rho t)} a^{(j)}(t,x), \quad j=0,1,2
	$$
	with the phase function 
	$$
	\Theta_j(x)=  \omega_j \cdot (x-p) +\frac12  \mathcal{H}_j(x) (x-p)\cdot(x-p)+O({\textrm{dist}(x,\gamma_{p,\omega})}^3).
	$$
The amplitude functions $a^{(j)}(t,x)|_{\p\Omega}$ are supported in $\Gamma$ given $\eta<\eta_0$ for some $\eta_0>0$ and the remainder functions $r_j(t,x)$ satisfy \eqref{remdS 0}. 
Let $f_j:=U_j|_{\p\Omega}$ ($j=0,1,2$). 
From $\Lambda_{q,\beta_1}(f)=\Lambda_{q,\beta_2}(f)$ on $\mathcal S_\lambda(\Sigma)$ with $\supp(f)\subset\Sigma^\sharp$ and Lemma \ref{lemma:ID}, we obtain the integral identity \eqref{ID:local uniqueness}. Plugging in above $U_j$ ($j=0,1,2$), we obtain
	\begin{align*} 
		0=\int_Q \beta U_1U_2\overline{U}_0\,dxdt=\int_Q \beta v_1v_2\overline{v}_0\,dxdt+ R_1+R_2+R_3,
	\end{align*}
	where the remainder terms are grouped as
	$$
	R_1:= \int_{Q} \beta (\overline{r}_0v_1v_2+r_1v_2\overline{v}_0+r_2v_1\overline{v}_0 ) \,dx dt,
	$$
	$$
	R_2:= \int_{ Q} \beta  (\overline{r}_0r_1 v_2+\overline{r}_0 r_2 v_1+r_1r_2\overline{v}_0 ) \,dx dt,
	$$
	$$
	R_3:= \int_{ Q} \beta  r_1r_2  \overline{r}_0  \,dx dt.
	$$
	When $2\kappa>{n+1\over 2}$, Proposition~\ref{prop:WKB solutions} shows that 
	\begin{align*} 
		 R_1+R_2+R_3 = O(\rho^{-K}) 
	\end{align*}
	for a large $K>\frac n2$ by choosing $N$ sufficiently large.
 
Note that $|\omega_1|^2+|\omega_2|^2=|\omega_0|^2$. The phase of the product is then given by
	\begin{align*}
		\Theta_1(x)+\Theta_2(x)-\overline{\Theta}_0(x)
		&=\frac12 \mathcal{H} (x) (x-p)\cdot (x-p)+\widetilde h(x),
	\end{align*}
	where $\mathcal{H}(x):=\mathcal{H}_1(x)+\mathcal{H}_2(x)-\overline{ \mathcal{H}_0}(x)$ whose imaginary part  
	$$\mathrm{Im}\mathcal{H}(x)=\mathrm{Im}\mathcal{H}_1(x)+\mathrm{Im}\mathcal{H}_2(x)+\mathrm{Im}\mathcal{H}_0(x).$$
By \eqref{eqn:H_pos_def}, we have
\[\frac12 \mathrm{Im}\mathcal{H} (x) (x-p)\cdot (x-p)\geq{c_0}(\textrm{dist}(x,\gamma_{p,\omega_1})^2+\textrm{dist}(x,\gamma_{p,\omega_2})^2)\geq {c_0}|x-p|^2,\]
which implies $\mathrm{Im}\mathcal{H}$ is positive definite.
Also, we have for $|x-p|$ small,
\begin{equation}\label{eqn:h_rho}|\widetilde h(x)|=O(\textrm{dist}(x,\gamma_{p,\omega_1})^3+\textrm{dist}(x,\gamma_{p,\omega_2})^3+\textrm{dist}(x,\gamma_{p,\omega_0})^3)=O(|x-p|^3).\end{equation}
Therefore, for $\eta<\eta_0$ sufficiently small, we shall have
\begin{equation}\label{eqn:ImTheta_pos}
\mathrm{Im}(\Theta_1(x)+\Theta_2(x)-\overline{\Theta}_0(x))\geq \widetilde c_0|x-p|^2\qquad \textrm{ when } |x-p|<\eta.
\end{equation}

Finally, standing on these, we derive 
{
	\begin{align*}
		O(\rho^{-K} )=&\int_Q \beta v_1v_2\overline{v}_0\,dxdt\\
		=& \int_0^T\int_{B_{ 2\eta}(p)}  \beta e^{i\rho (\frac12   \mathcal{H}(x)  (x-p)\cdot (x-p)+\widetilde h(x))} (\widetilde a_0(t,x)+O(\rho^{-1}))\widetilde\chi_\eta(x) \,dxdt , \notag
	\end{align*}
where $\widetilde a_0(t,x)=a^{(1)}_0 a^{(2)}_0 \overline{a}^{(0)}_0(t,x)$ and $\widetilde\chi_\eta(x):=\prod_{j=0,1,2}\chi_\eta{(z'_{p,\omega_j}(x))}$ with $z'_{p,\omega_j}(x)$ being the projection of $x-p$ onto the orthogonal $(n-1)$-dim subspace $\omega_j^\perp=\{\xi\in\R^n:~\xi\cdot \omega_j=0\}$. 
By the change of variable $\tilde{x}=\rho^{1\over 2}(x-p)$, we have 
\[O(\rho^{-K+\frac n2})=
 \int_0^T  \int_{B_{2\eta\sqrt{\rho}}(0)} e^{i (\frac12   \mathcal{H}(\rho^{- \frac12}\tilde{x}+p)\tilde{x}\cdot \tilde{x}+\rho\widetilde h(\rho^{- \frac12}\tilde{x}+p))}  (\beta\widetilde a_0(t, \rho^{- \frac12}\tilde{x}+p ) +O(\rho^{-1}))\widetilde\chi_\eta(\rho^{- \frac12}\tilde{x}+p) \,d\tilde{x}dt.\]
Applying \eqref{eqn:h_rho} and \eqref{eqn:ImTheta_pos}, and by the dominated convergence theorem, we obtain the limit as $\rho\rightarrow\infty$
	\begin{align*}
		\LC \int_{\R^{n }} e^{\frac{i}{2}  \mathcal{H}{(p)}\tilde{x} \cdot\tilde{x}} \,d\tilde{x} \RC \LC\int_0^T  \beta \widetilde a_0 (t,p) \,dt\RC =0,
	\end{align*}
where we use that the pointwise limit of $\rho\widetilde h(\rho^{-1/2}\widetilde x+p)$ is zero. We can choose the initial condition for $H$ in the matrix Riccati equation such that the first integral is nonzero.}
Also recall that $a^{(j)}_0(t,p)=\iota(t)$ for $j=0,1,2$ in the constructions of $a^{(j)}_0$, thus
	$$
	\int_0^T \beta(t,p) \iota^3(t) \,dt=0.
	$$
	Since $\iota$ can be chosen to be any smooth cut-off function at the time variable, this leads to $\beta(t,p)=0$ for arbitrary $t\in (0,T)$. 
\end{proof}


\section{Proof of Theorem \ref{thm:stability} and Theorem \ref{thm:unique}}\label{sec:stability}

\subsection{Geometric optics}\label{sec:Go}

 In this section, we will construct the geometric optics (GO) solutions to the Schr\"odinger equation, similar to the ones used in \cite{KianSoccorsi} and \cite{LOST2022}, and introduce its associated unique continuation principle. Compared to the GO solutions in Proposition \ref{prop:WKB solutions}, these are not localized near a straight line.

Following the same ansatz for a GO solution under the global coordinate
\begin{align*}
	u(t,x) = e^{i\Phi(t,x)} \LC \sum^{N}_{k=0}\rho^{-k}a_k(t,x)\RC + r(t,x),
\end{align*}
where we take a simple linear (in $x$) phase  
$$
    \Phi(t,x):= \rho(x\cdot \omega-\rho |\omega|^2t),
$$ 
with $\rho>0$ and $\omega\in\R^n$. Then the terms in the amplitude naturally satisfy 
\begin{equation}\label{CON:a_k}\begin{split}
	&\omega\cdot\nabla a_0=0,\\
	&2i\omega\cdot\nabla a_1=- (i\p_t + \Delta + q)a_{0},\\
	&\qquad\qquad\vdots\\
	&2i\omega\cdot\nabla a_N=- (i\p_t + \Delta + q)a_{N-1},
\end{split}
\end{equation}
and the remainder term $r$ satisfies
\begin{align}\label{CON:r}
	\left\{\begin{array}{rcll}
	(i\p_t + \Delta + q) r &=&- \rho^{-N}e^{i\Phi(t,x)}(i\p_t + \Delta + q)a_N&\quad \hbox{in }Q,\\
	r&=&0 &\quad \hbox{on }\Sigma,\\
	r&=&0 &\quad \hbox{on }\{0\}\times \Omega.\\
\end{array} \right. 
\end{align}

We construct $a_0$ as follows. Let $0<T^*<T$ and the function $\theta_h\in C^\infty_0(\R)$ satisfy $0\leq \theta_h\leq 1$ and for $0<h<{T^*\over 4}$,
\begin{align}\label{eqn:theta_h}
	\theta_h(t) =\left\{  
	\begin{array}{ll}
		0 & \hbox{in }[0,h]\cup[T^*-h,T^*],\\
		1 & \hbox{in }[2h,\,T^*-2h],
	\end{array} 
	\right.
\end{align}
with support in $(h,T^*-h)$ and, moreover, for all $j\in \mathbb{N}$, there exist constants $C_j>0$ such that 
\begin{align}\label{EST:theta}
\|\theta_h\|_{W^{j,\infty}(\R)}\leq C_jh^{-j}.
\end{align}
We choose 
$$
a_0(t,x) :=\theta_h(t)e^{ i(t\tau+x\cdot\xi)}
$$
with $\xi\in \omega^\perp$. Then it satisfies 
$$
a_0(t,x) =0\quad\hbox{ for all }(t,x)\in ((0,h)\cup(T^*-h,T^*))\times\Omega,
$$
and the first equation in \eqref{CON:a_k}.

Let $y\in\p\Omega$ and $L:=\{x:\omega\cdot(x-y)=0\}$. Set
\begin{align}\label{eqn:a_k_GO_4}
	a_k(t,x+s\omega) = {i\over 2} \int^{s}_0 (i\p_t + \Delta + q)a_{k-1} (t,x+\tilde{s}\omega)\, d\tilde{s}, \qquad x\in L,\quad j=1, \ldots, N.
\end{align} 
Then $a_j$ ($j=1, \ldots, N$) satisfies \eqref{CON:a_k} and vanishes on $L$. The regularity of $a_j$ inherits from $a_0$, which is smooth both in $t$ and $x$.

We introduce the notation
$$
\langle \tau,\xi\rangle : =(1+\tau^2+|\xi|^2)^{1/2}, \quad  \tau \in \R,\,\xi\in\R^{n}.
$$ 
\begin{proposition}\label{prop:GO existence}
	Let $\omega\in\R^n$, $N>0$ and $m>\frac{n+1}2$ be an integer. Suppose that $q\in C^\infty(Q)$.
	Then there exist GO solutions to the Schr\"odinger equation  $(i\p_t + \Delta + q)u=0$ in $Q$ of the form
	$$
	u(t,x) = e^{i\Phi(t,x)}\LC a_0(t,x)+\sum^{N}_{k=1}\rho^{-k}a_k(t,x)\RC + r(t,x),\quad a_0(t,x) =\theta_h(t)e^{ i(t\tau+x\cdot\xi)}
	$$
    satisfying the initial condition $u|_{t=0}=0$ in $\Omega$ (or the final condition $u|_{t=T}=0$ in $\Omega$). Here $a_k\in H^{m}(Q)$ ($k=1, \ldots, N$) are given by \eqref{eqn:a_k_GO_4} and satisfy
	 \begin{align}\label{EST:ak}
	 \|a_k\|_{H^{m}(Q)}\leq C \tx^{2k+m} h^{-k-m},~~0\leq k\leq N 
	 \end{align}
	 for any $\tau\in\R$, $h\in(0,\frac{T^*}4)$ small enough and $\xi \in \omega^\perp$, where the constant $C>0$ depending only on $\Omega$ and $T$. 
	 The remainder term $r$ satisfies
	\begin{align}\label{EST:r}
	\|r\|_{H^m(Q)}\leq C\rho^{-N + 2m}\langle \tau,\xi\rangle^{2N+m+2}h^{-(N+m+1)} 
	\end{align}
	and
	\begin{align*} 
    \|r\|_{C^1([0,T],L^2(\Omega))\cap C([0,T],H^2(\Omega)) }\leq C \rho^{-N+2 } \tx^{ 2N+2} h^{- N-2} 
  \end{align*}
	for some constant $C>0$ depending only on $\Omega$ and $T$. 
\end{proposition}
\begin{proof}
We show the proof for the case with zero initial condition. The case with zero final condition at $T$ can be justified similarly. 
For $k=0$, the estimate \eqref{EST:ak} clearly holds for $m=0$. 
For $m=1$, it is easy to check that $\|\nabla a_0\|_{L^2(Q)}\leq C|\xi|$ and $\|\p_t a_0\|_{L^2(Q)}\leq C(|\tau|+|\xi|+h^{-1})$ and, therefore, when $h$ is small,
    $$
    \|a_0\|_{H^1(Q)}\leq C\langle \tau,\xi\rangle h^{-1}.
    $$
    Similarly, we can also deduce the bound for $\|a_0\|_{H^m(Q)}$.
    By induction, assuming that $a_{k-1}$ satisfies 
    $$
    \|a_{k-1}\|_{H^{m}(Q)}\leq C \tx^{2k+m-2} h^{-k-m+1 }.
    $$
    From \eqref{eqn:a_k_GO_4}, since we take $x$-derivative twice and $t$-derivative on $a_{k-1}$, the estimate of $\|a_k\|_{H^{m}(Q)}$ will receive extra $\tx^2$ and $h^{-1}$ on top of $\|a_{k-1}\|_{H^{m}(Q)}$. This leads to \eqref{EST:ak}. Note that \eqref{EST:ak} holds for all integer $m\geq 0$.

	Now we discuss the existence and estimates of $r$ to the problem \eqref{CON:r}.
   From Proposition~3 and Lemma~4 in \cite{LOST2022}, since $e^{i\Phi}(i\p_t+\Delta+q)a_N\in \mathcal{H}^{m}_0$ with even integer $m>\frac{n+1}2$, there exists a solution $r$ to \eqref{CON:r} so that
	$$
	    \|r\|_{H^m(Q)}\leq C \rho^{-N} \|e^{i\Phi}(i\p_t+\Delta+q)a_N\|_{H^m(Q)}\leq C\rho^{-N+2m}\langle \tau,\xi\rangle^{2N+m+2}h^{-(N+m+1)}.
	$$
	In addition, from Lemma~2.3 in \cite{KianSoccorsi}, one can also derive
	\begin{align*}
    \|r\|_{C^1([0,T],L^2(\Omega))\cap C([0,T],H^2(\Omega)) }
   \leq  & \,C \rho^{-N} \|e^{i\Phi}(i\p_t+\Delta+q)a_N\|_{H^1(0,T,L^2(\Omega))}\\
   \leq &\,C \rho^{-N} \rho^2 ( \| a_N  \|_{H^2(0,T,L^2(\Omega))}+\| a_N  \|_{H^1(0,T,H^2(\Omega))}    )\\
      \leq &\, C \rho^{-N+2 } \tx^{ 2N+2} h^{- N-2}.
   \end{align*}
\end{proof}

\begin{remark}\label{rk:GO}
	The choice of $a_0$ is quite flexible as long as $\omega\cdot\nabla a_0=0$ is fulfilled. This flexibility is essential in the reconstruction of the unknown coefficient $\beta$ since it will help eliminate the unwanted terms in the integral identity in order to obtain the Fourier transform of $\beta$, see Section~\ref{sec:theorem 1.2} for more detailed computations and explanations.
 
    For our purpose, we will also need the GO solution with a simple choice $a_0(t,x)=\theta_h(t)$ where $\theta_h(t)$ is given by \eqref{eqn:theta_h}.  
	That is, there exist GO solutions to the Schr\"odinger equation  $(i\p_t + \Delta + q)u=0$ in $Q$ of the form
\[
	u(t,x) = e^{i\Phi(t,x)}\LC \theta_h(t)+ \sum^{N}_{k=1}\rho^{-k}a_k(t,x)\RC + r(t,x),
\]
satisfying the initial condition $u|_{t=0}=0$ in $\Omega$ (or the final condition $u|_{t=T}=0$ in $\Omega$).
From \eqref{CON:a_k}, we obtain $\omega\cdot\nabla a_1=-\frac12\p_t \theta_h(t)$, implying $$a_1(t,x)=-\frac1{2|\omega|^2}\p_t\theta_h(t)x\cdot \omega$$
with $a_1(t,x)=0$ on $\omega^\perp$.
The rest of $a_k\in H^{m}(Q)$ ($k=2,\ldots, N$) are given by \eqref{eqn:a_k_GO_4} and one can verify
	 \begin{align}\label{EST:ak 1}
	 \|a_k\|_{H^{m}(Q)}\leq C h^{-k-m},~~0\leq k\leq N 
	 \end{align}
	 and
	\begin{align}\label{EST:r 1}
	\|r\|_{H^m(Q)}\leq C\rho^{-N + 2m} h^{-(N+m+1)},
	\qquad
    \|r\|_{C^1([0,T],L^2(\Omega))\cap C([0,T],H^2(\Omega)) }\leq C \rho^{-N+2 }  h^{- N-2}
  \end{align}
	for some constant $C>0$ depending only on $\Omega$ and $T$.
	Note that under this construction $a_k$ ($k=0,\cdots, N$) all vanish on $((0,h)\cup(T-h,T))\times\Omega$.
\end{remark}

\subsection{Unique continuation property (UCP)}
Recall that $\mathcal{O}\subset\Omega$ is an open neighborhood of $\p\Omega$. Let $\mathcal{O}_j$ ($j=1,2,3$) denote the open subsets of $\mathcal{O}$ such that
$ 
\overline{\mathcal{O}}_{j+1} \subset \mathcal{O}_j,~~\overline{\mathcal{O}}_{j  }\subset \mathcal{O}. 
$ 
Set $ \Omega_j:=\Omega\setminus \overline{\mathcal{O}}_{j}$ and $Q_j:=(0,T)\times \Omega_j.$ We will need the following lemma of UCP and its corollary for the linear Schr\"odinger equation. The lemma follows directly from \cite{Bella-Fraj2020} by setting the magnetic potential to be zero. 
\begin{lemma}[Unique continuation property]\label{lemma:UCP} 
Suppose that $q\in \mathcal M_{\mathcal O}$. Let $\tilde{w}\in H^{1,2} (Q)$ be  a solution to the following system
\begin{align}\label{UCP linear W}
	\left\{\begin{array}{rcll}
		(i \p_t +\Delta + q)\tilde{w}&=&  g_0 &\quad \hbox{in }Q,\\
		\tilde{w} &=& 0&\quad \hbox{on }\Sigma,\\
		\tilde{w} &=&0 &\quad \hbox{on }\{0\}\times \Omega,\\
	\end{array} \right. 
\end{align} 
where $g_0\in L^2(Q)$ and $\mathrm{supp} (g_0)\subset (0,T)\times (\Omega\setminus \mathcal{O})$. Then for any $T^{*}\in (0,T)$, there exist $C>0, \gamma^{*}>0, m_1>0,\mu_1<1$ such that the following estimate holds
$$
\|\tilde{w}\|_{L^2((0,T^{*})\times (\Omega_3\setminus \Omega_2))}\leq C\LC\gamma^{-\mu_1} \|\tilde{w}\|_{H^{1,1}(Q)}+e^{m_1\gamma} \|\p_{\nu} \tilde{w}\|_{L^2(\Sigma^{\sharp})} \RC,
$$
for any $\gamma> \gamma^{*}$. Here the constants $C,\, m_1$ and $\mu_1$  depend on $\Omega,\, \mathcal{O},\,T^{*}$ and $T.$  
\end{lemma}

\begin{corollary}
Let $q\in \mathcal M_{\mathcal O}$, and $\tilde{w}\in H^{1,2}(Q)$ a solution of \eqref{UCP linear W} where $g_0\in L^2(Q)$ and   $\mathrm{supp} (g_0)\subset (0,T)\times (\Omega\setminus \mathcal{O})$ such that $\p_{\nu} \tilde{w}=0$ on $\Sigma^{\sharp}$. Then $\tilde{w}=0$ in $(0,T)\times (\Omega_3\setminus \Omega_2)$. 
\end{corollary}


\subsection{The integral identity}
In this section, we derive the needed integral identity to prove the stability estimate in Theorem \ref{thm:stability}. We denote $$Q^*:=(0,T^*)\times\Omega\quad  \hbox{ for $0<T^*<T$}.$$
Recall the notation $u_{\ell,\varepsilon f}$ ($\ell=1,2$) that denotes the small unique solution to the initial boundary value problem
\begin{align*} 
	\left\{\begin{array}{rcll}
		(i \p_t +\Delta +q)u_{\ell,\varepsilon f} + \beta_\ell u_{\ell,\varepsilon f}^2 &=& 0 &\quad \hbox{in }Q^*,\\
		u_{\ell,\varepsilon f} &=&\varepsilon_1f_1+\varepsilon_2f_2 &\quad \hbox{on }\Sigma,\\
		u_{\ell,\varepsilon f} &=&0 &\quad \hbox{on }\{0\}\times \Omega,\\
	\end{array} \right. 
\end{align*} 
where $f_1, f_2\in  H^{2\kappa+\frac32, 2\kappa+\frac32}(\Sigma)$ and $|\varepsilon|:=|\varepsilon_1|+|\varepsilon_2|$ is sufficiently small such that $\varepsilon_1f_1+\varepsilon_2f_2\in\mathcal S_\lambda(\Sigma)$. Also, let $U_{j}$ and $W_{\ell, (1,1)}$ be the solutions to the equations \eqref{IBVP linear U} and \eqref{IBVP linear W}, respectively.
In addition, let $U_0$ be the solution of the adjoint problem, 
\begin{align*} 
	\left\{\begin{array}{rcll}
		(i \p_t +\Delta + q)U_0 &=& 0 &\quad \hbox{in }Q^*,\\
		U_0&=& f_0 &\quad \hbox{on }\Sigma,\\
		U_0 &=&0 &\quad \hbox{on }\{T^*\}\times \Omega\\
	\end{array} \right. 
\end{align*} 
for some $f_0\in H^{2\kappa+\frac32, 2\kappa+\frac32}(\Sigma)$ to be specified later. 
We also introduce a smooth cut-off function $\chi\in C^\infty(\overline{\Omega})$ satisfying $0\leq \chi\leq 1$ and  
\begin{align*}
	\chi(x) =\left\{  
	\begin{array}{ll}
		0 & \hbox{in }\mathcal{O}_3,\\
		1 & \hbox{in }\overline\Omega\setminus \mathcal{O}_2,
	\end{array} 
	\right.
\end{align*}
and denote 
$
 W:=W_{2,(1,1)}-W_{1,(1,1)}
$,
which solves
	\begin{align}\label{IBVP linear tilde W}
		\left\{\begin{array}{rcll}
			(i \p_t +\Delta + q)W&=& 2\beta U_1U_2 &\quad \hbox{in }Q^*,\\
			W &=& 0&\quad \hbox{on }\Sigma,\\
			W &=&0 &\quad \hbox{on }\{0\}\times \Omega,\\
		\end{array} \right. 
	\end{align} 
where $\beta=\beta_1-\beta_2$.
As we will see below, by applying this cut-off function $\chi$ to $W$, whose Neumann data is not necessary zero, we have a control of the energy near the boundary using UCP.
First, we obtain the following key integral identity.
\begin{lemma} 
\label{lemma:identity}
Suppose that $\beta :=\beta_1-\beta_2\in \mathcal{M}_{\mathcal{O}}$.  
Let $U_j$ and $W$ be as above.
	Then
	\begin{align}\label{Integral identity}
	\int_{Q^*} 2\beta U_1U_2\overline{U}_0 + [\Delta,\chi]W\overline{U}_0\,dxdt= 0, 
	\end{align}
where $[\Delta, \chi]:=\Delta \chi -\chi \Delta$ is the commutator bracket.
\end{lemma}

\begin{proof}
Let $W^*(t,x):=\chi(x) W(t,x)$. Note that since $\beta_1-\beta_2=0$ in $[0,T]\times \mathcal{O}$ and $\chi=1$ in $\overline\Omega\setminus \mathcal{O}$ (a subset of $\overline\Omega\setminus \mathcal{O}_2$), 
we have
$$
    \chi(\beta_1-\beta_2)=\beta_1-\beta_2\quad\hbox{ in }Q.
$$ 
This implies that the function $W^*$ satisfies
\begin{align*} 
	\left\{\begin{array}{rcll}
		(i \p_t +\Delta + q)W^*&=& 2\beta U_1U_2 + [\Delta,\chi]W &\quad \hbox{in }Q^*,\\
		W^* &=& 0&\quad \hbox{on }\Sigma,\\
		W^* &=&0 &\quad \hbox{on }\{0\}\times \Omega.\\
	\end{array} \right. 
\end{align*} 
In particular, we have
$$
    W^*|_{\Sigma}=\p_\nu W^*|_{\Sigma}=0.
$$
We multiply the first equation in \eqref{IBVP linear W} by $\overline{U}_0$ and then integrate over $Q^*$.  Using the condition $\overline{U}_0|_{t=T^*}=W|_{t=0}=0$,  we finally obtain
\begin{align*}
	\int_{Q^*} 2 \beta U_1U_2 \overline{U}_0 +[\Delta,\chi]W\overline{U}_0\,dxdt
	=\int_\Sigma (\overline{U}_0\p_\nu W^* - W^*\p_\nu\overline{U}_0)\,d\sigma(x) dt=0.
\end{align*}
\end{proof}

\subsection{Proof of the stability estimate (Theorem \ref{thm:stability})}\label{sec:theorem 1.2}
Below we derive a series of estimates to prove the final stability result in Theorem \ref{thm:stability}.
We choose to plug in GO solutions $U_j$, $j=0,1,2$ as in Proposition \ref{prop:GO existence} and Remark \ref{rk:GO}. Specifically, we take 
	\begin{align*} 
		U_j(t,x) :=v_j(t,x)+r_j(t,x)= e^{i\Phi_j(t,x)} \LC a_0^{(j)} + \sum^{N}_{k=1}\rho^{-k}a_{k}^{(j)}(t,x)\RC + r_j(t,x),\qquad j=0,1,2,
	\end{align*}
where the phase function $\Phi_j$ are of the form
	\begin{align*} 
	    \Phi_j(t,x) = \rho\LC x\cdot\omega_j-\rho|\omega_j|^2 t \RC
	\end{align*}
	with the vectors $\omega_1,\,\omega_2$ and $\omega_0$ satisfying
    \begin{align}\label{CON_omega}
    	\omega_1+\omega_2=\omega_0,\quad |\omega_1|^2+|\omega_2|^2 = |\omega_0|^2.
    \end{align}
The leading amplitudes $a^{(j)}_0$ are given by 
\begin{align*} 
	    a^{(1)}_0(t,x) =a^{(2)}_0(t,x)=\theta_h(t),\quad a^{(0)}_0(t,x)=\theta_h(t) e^{i(\tau t+x\cdot\xi)},
	\end{align*}
where $\tau\in \R$ and $\xi\in \omega^\perp_0$.

Substituting $U_j=v_j+r_j$ ($j=0,1,2$) into the first term on the left-hand side of the identity \eqref{Integral identity}, we get 
\begin{align}\label{integral U1U2U0}
  &\int_{Q^*} 2\beta U_1U_2\overline{U}_0 \,dx dt = \int_{Q^*} 2\beta  v_1v_2 \overline{v}_0 \,dx dt + R_1+R_2+R_3,
\end{align}
where the remainder terms are grouped into
$$
    R_1:=2\int_{Q^*} \beta (\overline{r}_0v_1v_2+r_1v_2\overline{v}_0+r_2v_1\overline{v}_0 ) \,dx dt,
$$
$$
    R_2:=2\int_{Q^*} \beta  (\overline{r}_0r_1 v_2+\overline{r}_0 r_2 v_1+r_1r_2\overline{v}_0 ) \,dx dt,
$$
$$
    R_3:= 2\int_{Q^*} \beta  r_1r_2  \overline{r}_0  \,dx dt.
$$
We have the following asymptotics.
\begin{lemma}\label{lemma:v1v2v0} 
Let $m>{n+1\over 2}$.
There exist $\rho_0>1$ and $1>h_0>0$ such that for $\rho>\rho_0$ and $0<h<h_0$ such that 
\begin{align}\label{integral v1v2v0}
   2 \int_{Q^*} \beta  v_1v_2 \overline{v}_0 \,dx dt = 2\int_{Q^*} \beta  a^{(1)}_0a^{(2)}_0  \overline{a}^{(0)}_0 \,dx dt+I,
\end{align} 
where 
\begin{align*}
     |I|\leq C \LC\rho^{-1} \tx^{2N}h^{-N} + \rho^{-2} \tx^{2N}h^{-2N}+ \rho^{-3} \tx^{2N+m}h^{-3N-3m}\RC,
\end{align*}
for any $\tau\in\R$ and $\xi\in \omega_0^\perp$. Here the positive constant $C$ depends on $Q^*,\, N$, and $\beta$.
\end{lemma}
\begin{proof}
By the definition of $v_j$, we have the identity
$$
2\int_{Q^*} \beta  v_1v_2 \overline{v}_0 \,dx dt = 2\int_{Q^*} \beta  a^{(1)}_0a^{(2)}_0 \overline{a}^{(0)}_0 \,dx dt+I_1+I_2+I_3,
$$
where we used the conditions \eqref{CON_omega} to get $\Phi_1+\Phi_2-\Phi_0=0$. 
Here the rest $O(\rho^{-1})$ terms are grouped into
$$
    I_1:=2\int_{Q^*} \beta  \left[ a^{(1)}_0a^{(2)}_0 \LC\sum^{N}_{k=1}\rho^{-k}\overline{a}^{(0)}_{ k}\RC + a^{(1)}_0\overline{a}^{(0)}_0  \LC\sum^{N}_{k=1}\rho^{-k}a_{k}^{(2)}\RC +a^{(2)}_0\overline{a}^{(0)}_0 \LC \sum^{N}_{k=1}\rho^{-k}a_{k}^{(1)}\RC\right] \,dx dt,
$$
\begin{align*}
    I_2&:=2\int_{Q^*} \beta  \Bigg[ a_0 ^{(1)}\LC \sum^{N}_{k=1}\rho^{-k}a^{(2)}_{ k}\RC\LC \sum^{N}_{k=1}\rho^{-k}\overline{a}^{(0)}_{ k}\RC +a^{(2)}_0 \LC \sum^{N}_{k=1}\rho^{-k}a^{(1)}_{ k} \RC\LC \sum^{N}_{k=1}\rho^{-k}\overline{a}^{(0)}_{ k}\RC \\
    &\quad + \overline{a}^{(0)}_0 \LC \sum^{N}_{k=1}\rho^{-k}a^{(1)}_{ k} \RC\LC \sum^{N}_{k=1}\rho^{-k}a^{(2)}_{ k} \RC \Bigg]\,dx dt,
\end{align*}
and
$$
    I_3:= 2\int_{Q^*} \beta \LC \sum^{N}_{k=1}\rho^{-k}a^{(1)}_{ k} \RC\LC \sum^{N}_{k=1}\rho^{-k}a^{(2)}_{ k} 
    \RC \LC\sum^{N}_{k=1}\rho^{-k}\overline{a}^{(0)}_{ k}   \RC \,dx dt.
$$

Let us estimate each $I_j$. To this end, it is sufficient to control the first term in each $I_j$ since the other terms can be handled similarly. 

The first term in $I_1$ is controlled by
\begin{align*}
   &  2 \left| \int_{Q^*} \beta  a^{(1)}_0a^{(2)}_0 \LC \sum^{N}_{k=1}\rho^{-k}\overline{a}^{(0)}_{ k} \RC\, dx dt \right|\\
    \leq &\,C \|\beta  \|_{L^{\infty}(Q^*)} \|a^{(1)}_0  \|_{L^{\infty}(Q^*)} \|a^{(2)}_0 \|_{L^{\infty}(Q^*)}
    \LC \sum^{N}_{k=1}\rho^{-k} \| \overline{a}^{(0)}_{ k} \|_{L^{2}(Q^*)}\RC \\
    \leq  &\, C \sum^{N}_{k=1}\rho^{-k} \|  \overline{a}^{(0)}_{ k} \|_{L^{2}(Q^*)}\leq C\rho^{-1}\|  \overline{a}^{(0)}_{N} \|_{L^{2}(Q^*)} 
    \leq C \rho^{-1} \tx^{2N}h^{-N},
\end{align*}
by \eqref{EST:ak} for sufficiently large $\rho>1$ and small $h<1$, where $C$ depending on $Q^*,\, N$ and $\beta$. 
Similarly, the second term and the third term are less than $C\rho^{-1} h^{-N}$ by applying \eqref{EST:ak 1} instead.
Combining these estimates together gives 
\begin{align}\label{EST:I1}
    |I_1|\leq C \rho^{-1} \tx^{2N}h^{-N}.
\end{align}
 
For $I_2$, applying H\"{o}lder's inequality, the first term is controlled by
\begin{align*} 
    &2\left| \int_{Q^*} \beta  a^{(1)}_0\LC \sum^{N}_{k=1}\rho^{-k}a^{(2)}_{ k}\RC\LC \sum^{N}_{k=1}\rho^{-k}\overline{a}^{(0)}_{ k}\RC\, dx dt \right|\notag\\
    \leq & \, C \|\beta  \|_{L^{\infty}(Q^*)} \|a^{(1)}_0  \|_{L^{\infty}(Q^*)} \LC \sum^{N}_{k=1}\rho^{-k}\| a^{(2)}_{ k} \|_{L^{2}(Q^*)}\RC \LC\sum^{N}_{k=1}\rho^{-k} \|  \overline{a}^{(0)}_{ k} \|_{L^{2}(Q^*)}\RC \notag\\
    \leq &\, C\LC\sum^{N}_{k=1}\rho^{-k} h^{-k} \RC\LC \sum^{N}_{k=1}\rho^{-k}\tx^{2k} h^{-k}\RC\notag\\
    \leq &\, C \rho^{-2} \tx^{2N}h^{-2N},
\end{align*}
by using \eqref{EST:ak} and \eqref{EST:ak 1} again. Similarly, the second and the third terms share the same bound. Therefore we have
\begin{align}\label{EST:I2}
    |I_2|\leq C \rho^{-2} \tx^{2N}h^{-2N}.
\end{align}

Finally, since $m>{n+1\over 2}$, 
  we can control $I_3$ by
\begin{align}\label{EST:I3}
    |I_3|\leq &\, C \|\beta  \|_{L^{\infty}(Q^*)}\LC \sum^{N}_{k=1}\rho^{-k}\| a^{(1)}_{ k} \|_{H^{m}(Q^*)} \RC\LC \sum^{N}_{k=1}\rho^{-k}\|a^{(2)}_{ k} \|_{H^{m}(Q^*)}\RC \LC
    \sum^{N}_{k=1}\rho^{-k}\|\overline{a}^{(0)}_{ k} \|_{H^{m}(Q^*)}\RC \notag\\
    \leq &\, C \rho^{-3} \|a^{(1)}_N\|_{H^{m}(Q^*)} \|a^{(2)}_N\|_{H^{m}(Q^*)} \|a^{(0)}_N\|_{H^{m}(Q^*)} \notag\\
    \leq &\, C \rho^{-3} \tx^{2N+m}h^{-3N-3m}
 \end{align}
Combining \eqref{EST:I1}, \eqref{EST:I2}, and \eqref{EST:I3} completes the proof.
\end{proof}

\begin{lemma}\label{lemma:R terms} 
Let $m>{n+1\over 2}$. Then there exists $\rho_0>1$ and $0<h_0<1$ such that the three remainder terms satisfy the following estimates:
\begin{align*}
    |R_1|\leq C \rho^{-N+2m} \tx^{2N+m+2}h^{-3N-3m-1} ,
\end{align*}
\begin{align*}
    |R_2|\leq C \rho^{-2N+4m} \tx^{2N+m+2}h^{-3N-3m-2},  
\end{align*}
and
\begin{align*}
    |R_3|\leq C \rho^{-3N+6m} \tx^{2N+m+2}h^{-3N-3m-3}
\end{align*}
for $\rho>\rho_0$, $0<h<h_0$, $\tau\in\R$ and $\xi\in \omega_0^\perp$, where the positive constant $C$ depends on $Q^*,\, N$, and $\beta$.
\end{lemma}
\begin{proof}
Again it is sufficient to evaluate the first term in each $R_j$. Substituting $v_1,\,v_2$, and $r_0$ into the first term of $R_1$,  
we get
\begin{align*}
     \int_{Q^*} \beta \overline{r}_0 v_1v_2 dx dt= \int_{Q^*} \beta \overline r_0 e^{i \Phi_0} \LC \sum^{N}_{k=0}\rho^{-k}a^{(1)}_{ k}\RC\LC\sum^{N}_{k=0}\rho^{-k}a^{(2)}_{ k}\RC
       dx dt.
\end{align*}
Since $H^m(Q)$ is an algebra, by \eqref{EST:r} and \eqref{EST:ak 1}, we have
\begin{align*}
     \LV\int_{Q^*} \beta \overline{r}_0 v_1v_2 dx dt\RV \leq &  C \|\beta  \|_{L^{\infty}(Q^*)}\|r_0\|_{H^{m}(Q^*)} \LC \sum^{N}_{k=0}\rho^{-k}\|a^{(1)}_{ k} \|_{H^{m}(Q^*)}\RC\LC
     \sum^{N}_{k=0}\rho^{-k} \|a^{(2)}_{ k} \|_{H^{m}(Q)}\RC \\
     \leq &\, C   \rho^{-N+2m} \tx^{2N+m+2} h^{-N-m-1}  \|a^{(1)}_N\|_{H^{m}(Q^*)} \|a^{(2)}_N\|_{H^{m}(Q^*)}\\
     \leq &\, C \rho^{-N+2m} \tx^{2N+m+2}h^{-3N-3m-1}  .
\end{align*} 
The rest terms in $R_1$ satisfy the same estimate similarly.  
The same argument also gives the corresponding bounds for $R_2$ and $R_3$, using \eqref{EST:ak}, \eqref{EST:r}, \eqref{EST:ak 1} and \eqref{EST:r 1}.
This completes the proof of this lemma.
\end{proof}

 
Now we are ready to prove an estimate for the Fourier transform of $\beta \theta_h^3(t)$ below.
\begin{lemma}\label{lemma:fourier beta}
Let $2\kappa>{n+1\over 2}$, $N>4\kappa+1$ and $\beta=\beta_1-\beta_2\in \mathcal M_{\mathcal O}$. For $\rho>\rho_0>1$ and $1>h_0>h>0$, we have
\begin{align} \label{FOU:IN1}
    2\LV \int_{Q^*}\beta  \theta_h^3(t) e^{-i(\tau t+x\cdot \xi)} \,dx dt\RV
    \leq \LV\int_{Q^*}  [\Delta,\chi]W\overline{U}_0\,dxdt\RV + C \rho^{-1} \tx^{2N+2\kappa+2}h^{-3N-6\kappa-3} 
\end{align}
for $\tau\in\R$ and $\xi\in \omega^\perp_0$. Here the constant $C>0$ is independent of $\rho,\,\tau,\,\xi$ and $h$. 
\end{lemma}
\begin{proof}
We derive from \eqref{integral U1U2U0}, \eqref{integral v1v2v0} with $m=2\kappa$ and the identity \eqref{Integral identity} that
\begin{align}\label{IDT:IRR}
  2\int_{Q^*} \beta a^{(1)}_0a^{(2)}_0 \overline a^{(0)}_0 \,dx dt
  = -\int_{Q^*}  [\Delta,\chi]W\overline{U}_0\,dxdt- (I+R_1+R_2+R_3). 
\end{align}
With the estimates in Lemma~\ref{lemma:v1v2v0} and Lemma~\ref{lemma:R terms}, we can further simplify the estimate of $I+R_1+R_2+R_3$ into
$$
    |I+R_1+R_2+R_3|\leq C \rho^{-1} \tx^{2N+2\kappa+2}h^{-3N-6\kappa-3}  
$$
by noting that $N>4\kappa+1$, $\rho>1$, and $\tx\geq 1$.
The lemma is then proved by recalling that $a_0^{(0)}=\theta_h(t)e^{i(\tau t+x\cdot \xi)}$ and $a_0^{(1)}=a_0^{(2)}=\theta_h(t)$.
\end{proof}
Next we try to estimate the first term on the right hand side of \eqref{FOU:IN1} in terms of the boundary measurements difference.

\begin{lemma}\label{lemma:W} 
Let $2\kappa>{n+1\over 2}$ and $\beta=\beta_1-\beta_2\in \mathcal M_{\mathcal O}$. Suppose
$$\|(\Lambda_{q,\beta_1}-\Lambda_{q,\beta_2})f\|_{L^2(\Sigma^\sharp)}\leq \delta\qquad \textrm{ for all } f\in \mathcal S_\lambda(\Sigma).$$ 
Then for $\rho>\rho_0>1$, $1>h_0>h>0$ and $|\varepsilon_1|+|\varepsilon_2|$ sufficiently small, we have
\begin{align*} 
    \|W\|_{H^{1,1}(Q)}\leq C\rho^{8\kappa} h^{-2N-4\kappa-2}
\end{align*}
and 
\begin{align*} 
    \|\p_{\nu} W\|_{L^2(\Sigma^\sharp)} 
    \leq {C\over\varepsilon_1\varepsilon_2} \LC\delta + (\varepsilon_1+\varepsilon_2)^3\rho^{12\kappa+12} h^{- 3N-6\kappa-9}\RC.
\end{align*}
\end{lemma}
\begin{proof}
Recall that from Remark~\ref{rk:GO}, we can derive
\begin{align}\label{EST:U12}
\|U_j\|_{H^{2s}(Q)}\leq C\rho^{4s}  h^{- N-2s-1},\qquad j=1,2 
\end{align}
when $2s>\frac{n+1}2$. 
We first take $s=\kappa$. Since the non-homogeneous term of \eqref{IBVP linear tilde W} is $2\beta U_1U_2\in H^{2\kappa}$, 
applying Lemma~4 in \cite{LOST2022} yields that
\begin{align*} 
    \|W\|_{H^{1,1}(Q)}\leq \|W\|_{H^{2\kappa}(Q)} 
    \leq  C \|U_1\|_{H^{2\kappa}(Q)} \|U_2\|_{H^{2\kappa}(Q)}
    \leq C\rho^{8\kappa}  h^{-2N-4\kappa-2},
\end{align*}
where $C$ depends on $\beta$, $\Omega$ and $T$.

Below we will estimate $\p_{\nu} W=\p_\nu W_{2, (1,1)}-\p_\nu W_{1, (1,1)}$. From $f_j=U_j|_{\Sigma}$, according to \eqref{EST:U12} with $s=\kappa+1$ and Theorem~2.1 (the trace theorem) in \cite{Lion}, we obtain
\begin{align}\label{trace U}
\|f_j\|_{H^{2\kappa+{3\over 2},2\kappa+{3\over 2}}(\Sigma)}    
\leq C\|U_j\|_{H^{2\kappa+2}(Q)}\leq C \rho^{4\kappa+4} h^{-N-2\kappa-3},
\end{align}
for $j=1,\,2$, where the constant $C$ is independent of $f_j$.

Denote $\widetilde{\mathcal R}=\widetilde{\mathcal R}_2-\widetilde{\mathcal R}_1$ where $\widetilde{\mathcal R}_\ell$ is the remainder as in \eqref{eqn:tilde_R} for $u_{\ell, \varepsilon f}$ ($\ell=1,2$).
From \eqref{eqn:W11_L2bdry} and \eqref{trace U}, we obtain 
\begin{align*}
    &\quad \|\p_{\nu} W \|_{L^2(\Sigma^\sharp)} \\
    &\leq {1\over \varepsilon_1\varepsilon_2}  \|\widetilde\Lambda (\varepsilon_1 f_1-\varepsilon_2 f_2)- \widetilde\Lambda (\varepsilon_1 f_1) -  \widetilde\Lambda (\varepsilon_2 f_2)\|_{L^2(\Sigma^\sharp)} + {1\over \varepsilon_1\varepsilon_2}   \|\p_\nu\widetilde{\mathcal{R}}\|_{L^2(\Sigma^\sharp)}  \notag\\
    &\leq {3\over \varepsilon_1\varepsilon_2} \delta + C{1\over\varepsilon_1\varepsilon_2}(\varepsilon_1+\varepsilon_2)^3\left(\|f_1\|_{H^{2\kappa+\frac32, 2\kappa+\frac32}(\Sigma)}+\|f_2\|_{H^{2\kappa+\frac32, 2\kappa+\frac32}(\Sigma)}\right)^3 \notag\\
    &\leq {C\over\varepsilon_1\varepsilon_2} \LC\delta + (\varepsilon_1+\varepsilon_2)^3\rho^{12\kappa+12} h^{- 3N-6\kappa-9}\RC.
\end{align*}
where $\widetilde \Lambda:=\Lambda_{q,\beta_1}-\Lambda_{q,\beta_2}$.

\end{proof}

\begin{lemma}\label{Lemma:Estimate wu}
	Suppose that $q\in\mathcal M_{\mathcal O}$ and $\beta_1-\beta_2\in \mathcal M_{\mathcal O}$.  
	Then for $N>0$ large enough there exist $\gamma^*>0$, $m_1>0$, $\rho_0>1$ and $0<h_0<1$ such that   
	\begin{align*} 
		\LV\int_{ Q^*}   [\Delta,\chi]W\overline{U}_0\,dxdt\RV	 
		\leq C \tx^{2N+4}  h^{- 4N-6\kappa-12} \LC\gamma^{-\mu_1}\rho^{8\kappa+4}    + e^{m_1\gamma} \rho^{12\kappa+16} \LC \varepsilon^{-2}  \delta + \varepsilon\RC \RC.
	\end{align*}
	for $\gamma>\gamma^*$, $\tau\in\R$, $\xi\in\omega_0^\perp$, $\rho>\rho_0$ and $0<h<h_0$. 
	Moreover, for each $(\tau,\xi)\in \R^{n+1}$, 
	the Fourier transform of $\beta\theta_h^3$ (extended by zero outside $Q^*$) satisfies
	\begin{align}\label{FOU:IN2} 
		|\widehat{\beta  \theta_h^3}(\tau,\xi)|
		&\leq  C \Big(\rho^{-1} \tx^{2N+2\kappa+2}h^{-3N-6\kappa-3} +\tx^{2N+4} \gamma^{-\mu_1}\rho^{8\kappa+4}  h^{- 4N-6\kappa-12} \notag \\
		&\quad +\tx^{2N+4} e^{m_1\gamma} \rho^{12\kappa+16}  h^{- 4N-6\kappa-12}
\LC \varepsilon^{-2}  \delta + \varepsilon\RC 
  \Big).
	\end{align}
\end{lemma}
 
\begin{proof}
	We choose $\varepsilon_1=\varepsilon_2=: \varepsilon$.  
	From Lemma~\ref{lemma:W}, we obtain
	\begin{align*} 
		\|\p_{\nu} W\|_{L^2(\Sigma^\sharp)} 
		\leq C\LC \varepsilon^{-2} \delta +  \varepsilon\rho^{12\kappa+12} h^{- 3N-6\kappa-9}\RC.
	\end{align*} 
	By the UCP in Lemma~\ref{lemma:UCP}, there exist $\gamma^*>0$, $m_1>0$ and $\mu_1<1$ such that
	\begin{align*}
		&\quad \LV\int_{Q^*}  [\Delta,\chi]W\overline{U}_0\,dxdt\RV \\
		& \leq C\|[\Delta,\chi]W \|_{L^2(0,T^*;H^{-1}(\Omega_3\setminus\Omega_2))} \|\overline{U}_0\|_{L^2(0,T^*;H^{1}(\Omega))}\\
		&\leq C\|W\|_{L^2((0,T^*)\times(\Omega_3\setminus\Omega_2))}\|\overline{U}_0\|_{L^2(0,T^*;H^{2}(\Omega))}\\
		&\leq C\LC\gamma^{-\mu_1}\|W\|_{H^{1,1}(Q)} + e^{m_1\gamma}\|\p_\nu W\|_{L^2(\Sigma^{\sharp})}\RC \rho^4\tx^{2N+4} h^{-N-3} \\
		&\leq C\tx^{2N+4} \LC\gamma^{-\mu_1}\rho^{8\kappa+4}  h^{-3N-4\kappa-5}  + e^{m_1\gamma}\LC \varepsilon^{-2} \rho^{4}  h^{- N-3}\delta + \varepsilon\rho^{12\kappa+16} h^{- 4N-6\kappa-12}\RC \RC
	\end{align*}
	for any $\gamma>\gamma^*$.

	Together with Lemma~\ref{lemma:fourier beta} this leads to \eqref{FOU:IN2} for $\xi\in\omega_0^\perp$. Choosing enough $\omega_0$, this ends the proof.
\end{proof}
\begin{proof}[Proof of Theorem~\ref{thm:stability}] 

Let $\rho=\gamma^{\mu_1\over 8\kappa+5}$ so that 
$$
\rho^{-1}=\gamma^{-\mu_1}\rho^{8\kappa+4}.
$$
We denote $$\alpha_1:=4N+6\kappa+12,\quad \alpha_2:=2N+2\kappa+2, \quad \mu:={\mu_1\over 8\kappa+5}.$$

Then from \eqref{FOU:IN2},  
it is not hard to see
\begin{align}\label{FOU:IN3} 
	|\widehat{\beta  \theta_h^3}(\tau,\xi)|
	&\leq  C\tx^{\alpha_2}h^{-\alpha_1}  \Big( \gamma^{-\mu} + e^{m_2\gamma}( \varepsilon + \varepsilon^{-2}\delta)\Big),
\end{align}
with some index $m_2>m_1>0$. 
For a fixed $M>1$, by \eqref{FOU:IN3} and Plancherel theorem, we deduce
\begin{align*}
	\| \beta\theta_h^3\|^2_{ H^{-1}(\R^{n+1})} &= \int_{|(\tau,\xi)|\leq M} \tx^{-2 } |\widehat{\beta\theta_h^3}(\tau,\xi)|^2 d \tau d\xi + \int_{|(\tau,\xi)|> M} \tx^{-2 } |\widehat{\beta\theta_h^3}(\tau,\xi)|^2 d \tau d\xi\\
	&\leq C \LC\int_{|(\tau,\xi)|\leq M}  \tx^{2\alpha_2}h^{-2\alpha_1} \LC\gamma^{-2\mu}   +e^{2m_2\gamma} ( \varepsilon^2 + \varepsilon^{-4}\delta^2)\RC  \,d\tau d\xi  + M^{-2}  \|\beta\theta_h^3\|^2_{L^2( \R^{n+1}  )} \RC \\
	& \leq   CM^{2\alpha_2+n+1}h^{-2\alpha_1}\LC \gamma^{-2 \mu }   
	+ e^{2m_2\gamma}( \varepsilon^2 + \varepsilon^{-4}\delta^2)\RC  + C M^{-2}m_0^2  ,
\end{align*}
by recalling that $|\beta|\leq m_0$.
Thus, 
\begin{align*}
	\|\beta\theta_h^3\|_{ H^{-1}(\R^{n+1})}  
	\leq  CM^{\alpha_2+\frac{n+1}{2}}h^{-\alpha_1}\LC \gamma^{- \mu }   
	+ e^{ m_2\gamma}( \varepsilon + \varepsilon^{-2}\delta )\RC  + C M^{-1}.
\end{align*}
By interpolating and \eqref{EST:theta}, 
\begin{align*}
	\|\beta\theta_h^3\|^2_{ L^2(Q^*) }
	&\leq  \|\beta\theta_h^3\|_{  H^{-1}(Q^*) } \|\beta\theta_h^3\|_{ H^{1}(Q^*)   } 
	\leq  C   \|\beta\theta_h^3\|_{  H^{-1}(Q^*) } h^{-1}\\
	&\leq  CM^{\alpha_2+\frac{n+1}{2}}h^{-\alpha_1-1}\LC \gamma^{- \mu }   
	+ e^{ m_2\gamma}( \varepsilon + \varepsilon^{-2}\delta )\RC  + C M^{-1}h^{-1}. 
\end{align*}
In addition, we write 
$$
\beta=\beta\theta_h^3+ \beta(1-\theta_h^3).
$$
Note that $1-\theta_h^3=0$ in $[2h,T^*-2h]$, which leads to
\begin{align*}
	\| 1-\theta_h^3\|^2_{L^2(0,T^*)}\leq \int_{0}^{2h} (1-\theta_h^3)^2 dt+\int^{T^*}_{T^*-2h}(1-\theta_h^3)^2 dt\leq 4h.
\end{align*} 
Hence,
\begin{align*}
	\|\beta\|^2_{L^2(Q^*)} 
	&\leq   C(\|\beta\theta_h^3\|^2_{L^2(Q^*)}+\|\beta(1-\theta_h^3) \|^2_{L^2(Q^*)})\\
	&\leq  CM^{\alpha_2+\frac{n+1}{2}}h^{-\alpha_1-1}\LC \gamma^{- \mu }   
	+ e^{ m_2\gamma}( \varepsilon + \varepsilon^{-2}\delta )\RC  + C M^{-1}h^{-1} +C h.
\end{align*}
Choose $h<T^*/4$ satisfying $M^{-1}h^{-1}=h$ (i.e., $h=M^{-{1\over 2}}$) such that the last two terms above have the same order. This results in
\begin{align*}
	\|\beta\|^2_{L^2(Q^*)}  \leq C M^{\alpha_3} \LC\gamma^{- \mu} +e^{ m_2\gamma}( \varepsilon + \varepsilon^{-2}\delta )\RC  + CM^{-\frac12},
\end{align*}
where $\alpha_3:=\alpha_2+\frac{1}{2}{(\alpha_1+n+2)}$. We also further choose $M=\gamma^{\frac{\mu}{\frac12+\alpha_3}}$  
such that 
$$
\gamma^{- \mu} M^{\alpha_3}= M^{-\frac12},
$$
which implies that there exist constants $0<\mu'<1$ and $m_3>m_2>0$ such that
\begin{align}\label{EST:beta final}
	\|\beta\|^2_{L^2(Q^*)}  \leq C\LC e^{m_3\gamma}  \varepsilon^{-2}\delta + e^{m_3\gamma}  \varepsilon + \gamma^{-\mu'} \RC.
\end{align}
For $\delta\in (0,\min\{1,\, e^{-6m_3\gamma^*},\,\Lambda^{1\over 2}\})$ with $\Lambda>1$, 
we take
$$
    \varepsilon = {\lambda \over 4} \Lambda^{-1\over 6}\delta^{1\over 3}\quad\hbox{ and }\quad \gamma = {1\over 6 m_3} |\log(\delta)|.
$$
Then \eqref{EST:beta final} becomes
\begin{align*}
	\|\beta\|^2_{L^2(Q^*)}  \leq C \LC \delta^{1\over6} + |\log(\delta)|^{-\mu'}\RC.
\end{align*}
where $C$ depends on $\Omega,\,T,\, T^*$, $m_0$, and $\lambda$ and $\Lambda$.
 
\end{proof}

Now we verify the small condition in the well-posedness.

\begin{remark}\label{rmk:smalldelta} From the above proof, the parameters are defined by
	$$
	\rho=\gamma^{\mu},~~ M=\gamma^{\frac{\mu}{\alpha_3+{1\over 2}}},~~h=M^{-\frac12}=\gamma^{-\frac{\mu}{2 \alpha_3+1}},~~\gamma= {1\over 6m_3} |\log(\delta)|.~
	$$
	From \eqref{trace U}, for $j=1,2$, 
	\begin{align*} 
		\|f_j\|_{H^{2\kappa+{3\over 2},2\kappa+{3\over 2}}(\Sigma)}    
		\leq C \rho^{4\kappa+4} h^{-N-2\kappa-3}\leq C \gamma^{(4\kappa+4)\mu+\frac{1}{2\alpha_3+1} (N+2\kappa+3) \mu}\leq C   e^{m_3\gamma}.
	\end{align*}
    We took $\varepsilon_j= \varepsilon$ above. Due to $\delta<\Lambda^{1\over 2}$, it follows that
    $$
        |\varepsilon_j|\leq {\lambda \over 4} \Lambda^{-1\over 6}\delta^{1\over 3} < {\lambda \over 4},
    $$
    and
	\begin{align*}
		\|\varepsilon_1 f_1+\varepsilon_2 f_2\|_{H^{2\kappa+{3\over 2},2\kappa+{3\over 2}}(\Sigma)} 
		\leq C{\lambda \over 2} \Lambda^{-1\over 6}
  \delta^{1\over 3} e^{m_3\gamma} 
		= C{\lambda \over 2} \Lambda^{-1\over 6}\delta^{1\over 6}
        < C{\lambda \over 2}\Lambda^{-1\over 12}<\lambda,
	\end{align*}
	provided $\Lambda$ is sufficiently large.  
 Hence, the Dirichlet data $\varepsilon_1 f_1+\varepsilon_2 f_2 $ belongs to $\mathcal{S}_\lambda(\Sigma)$. This justifies the well-posedness and our procedures discussed above.    
\end{remark}

\subsection{Proof of Theorem~\ref{thm:unique}}
\begin{proof}[Proof of Theorem~\ref{thm:unique}]
From Theorem~\ref{thm:local uniqueness}, we obtain that $\beta_1=\beta_2$ in a neighborhood of $\Gamma$. Combining with the hypothesis that $\beta_1-\beta_2=0$ on $(0,T)\times \mathcal{O}'$ yields that $\beta_1-\beta_2=0$ near the boundary $\p\Omega$. Thus one can assume that $\beta=0$ in some open neighborhood $\mathcal{O}$ of $\p\Omega$. Applying the result in Theorem~\ref{thm:stability}, for any $T^*\in (0,T)$, we derive that 
$\beta_1=\beta_2$ in $(0,T^*)\times \Omega$ by letting $\delta\rightarrow 0$, which completes the proof. 
\end{proof}

 \begin{remark}\label{rmk:higher_nonlinearity}
Theorem \ref{thm:local uniqueness} and Theorem \ref{thm:stability} hold true for more general nonlinearity, such as $\beta(t,x) u^m$ or $\beta(t,x) |u|^{2m} u$.  
For the former case, the integral identity becomes 
$\int \beta U_1U_2\ldots U_m \overline{U}_0\,dx dt=0$, where $U_j$ is the solution to the linear equation. Like the setting $m=2$ discussed above, the vectors  $\omega_j$ in the phase functions of GO solutions are chosen to satisfy
\begin{align*}
	\omega_0=\omega_1+\ldots+\omega_m\quad\hbox{ and }\quad 
	|\omega_0|^2=|\omega_1|^2+\ldots+|\omega_m|^2
\end{align*}
so that the leading complex phase functions vanish eventually in the integral identity.
Once the phase functions are determined, following similar arguments in the proof of theorems lead to the unique and stable determination of $\beta$.

For the case of Gross-Pitaevskii equation with nonlinearity $\beta|u|^2u$ and the generalized $\beta|u|^{2m}u$, we can treat similarly to obtain the integral identity 
\[
    \int \beta U_1\overline{U_2}U_3\overline{U_4}\ldots U_{2m-1}\overline{U_{2m}}U_{2m+1}\overline{U_0}~dxdt=0
\]
and choose
\begin{align*}
    &\omega_1-\omega_2+ \omega_3-\omega_4+\ldots+\omega_{2m+1}-\omega_0=0\\
    &|\omega_1|^2-|\omega_2|^2+ |\omega_3|^2-|\omega_4|^2+\ldots+|\omega_{2m+1}|^2-|\omega_0|^2=0.
\end{align*}
We can choose $U_1, \overline{U_2}, U_{2m+1}$ and $ \overline{U}_0$ to be GO-solutions supported near four  straight lines $\gamma_{p,\omega_1}$, $\gamma_{p,\omega_2}$, $\gamma_{p,\omega_{2m+1}}$, and $ \gamma_{p,\omega_0}$, respectively,  and let $U_{j}$ and $U_{j+1}$ be GO-solutions supported near $\gamma_{p,\omega_1} $   for  $j=3,5,\ldots,2m-1$   so that their complex phases will cancel the other in pairs.
Hence, $\omega_1,\omega_2,\omega_{2m+1},\omega_0$ should satisfy  
\begin{align*}
\omega_0+\omega_2&=\omega_1+\omega_{2m+1},\\
|\omega_0|^2+|\omega_2|^2 &= |\omega_1|^2+|\omega_{2m+1}|^2,
\end{align*}
which can be achieved, for instance, by choosing
\begin{align*}
    &\omega_0=(1,-1,\ldots,0), \quad
	\omega_1=(\sqrt{1-r^2},-1,\ldots, r),\\
    &\omega_2=(\sqrt{1-r^2},\sqrt{1-r^2},0,\ldots,r),\quad \omega_{2m+1}=(1, \sqrt{1-r^2},\ldots,0 ),\quad 0<r<1.
\end{align*}
 
\end{remark}

\bibliographystyle{abbrv}
\bibliography{NLSref}

\end{document}